\documentclass[11pt,a4paper,headinclude,footinclude,fleqn,reqno]{amsart}                 
\usepackage[T1]{fontenc}                   
\usepackage[utf8]{inputenc}                 
\usepackage[english]{babel}       
\usepackage{graphicx}                      %
\usepackage[font=small]{quoting}            %
\usepackage{caption}  
     \usepackage{amsmath}
\usepackage{amssymb,amsmath,amsfonts,amscd,amsthm,pb-diagram,mathrsfs,amsmath,color}
\usepackage{amsbsy,bm}
\usepackage{float}               
 \usepackage[top=.8in,bottom=1.2in,left=1.45in,right=1.45in]{geometry}
\usepackage{verbatim}
\usepackage{color}
\usepackage{picture}
\usepackage{graphicx}
\usepackage{graphics}
\usepackage{comment}
\usepackage{hyperref}
\hypersetup{colorlinks,linkcolor={blue},citecolor={blue},urlcolor={red}}

\newtheorem{thm}{Theorem}[section] 
 
\newtheorem{cor}[thm]{Corollary} 
\newtheorem{prop}[thm]{Proposition} 
 
\newtheorem{defn}[thm]{Definition} 
\theoremstyle{definition} 
\newtheorem{rem}[thm]{Remark} 
\theoremstyle{remark}

\def\O{\Omega}
\def\e{\epsilon}
\def\S{\Sigma} 
\def\n{\nabla}

\def\p{\partial}

\def\a{\alpha}
\def\b{\beta}
\def\n{\nabla}

\def\o{\omega}
\def\O{\Omega}
\def\p{\partial}
\def\e{\epsilon}
\def\a{\alpha}
\def\b{\beta}
\def\g{\gamma}
\def\d{\delta}
\def\k{\kappa}

\def\s{\sigma}

\def\n{\nabla}
\def\<{\langle}
\def\>{\rangle}
\def\div{{\rm div}}

\def\n{\nabla}

\def\RR{\mathbb{R}}

\def\BB{\mathbb{B}}
\def\SS{\mathbb{S}}

\def\o{\omega}
\def\O{\Omega}
\def\p{\partial}
\def\e{\epsilon}
\def\a{\alpha}
\def\b{\beta}
\def\g{\gamma}
\def\d{\delta}

\def\s{\sigma}

\def\ep{\epsilon}

\def\W{\mathcal W}

\DeclareMathOperator{\Rm}{Rm}

{\left\{\begin{array}{@{}l@{}}}{\end{array}\right.}
\patchcmd{\abstract}{\scshape\abstractname}{\textbf{\abstractname}}{}{}
\makeatletter 
\def\@makefnmark{} 
\makeatother

\numberwithin{equation}{section}

\begin{document}
\title[]{Alexandrov-Fenchel inequality for convex hypersurfaces with capillary boundary in a ball}
\author{Liangjun Weng and Chao Xia}

\address{School of Mathematical Sciences, Shanghai Jiao Tong University, Shanghai, 200240, P. R. China}

\email{ljweng08@mail.ustc.edu.cn}

\address{School of Mathematical Sciences, Xiamen University, Xiamen, 361005, P. R. China}
\email{chaoxia@xmu.edu.cn}


 
\maketitle
\begin{abstract}
In this paper, we first introduce the quermassintegrals for convex hypersurfaces with  capillary boundary  in the unit Euclidean ball $\BB^{n+1}$ and derive its first variational formula. Then by using a locally constrained nonlinear curvature flow, which preserves the $n$-th quermassintegral and non-decreases the $k$-th quermassintegral, we obtain the Alexandrov-Fenchel inequality for convex hypersurfaces with capillary boundary in $\BB^{n+1}$. This generalizes the result in \cite{SWX} for convex hypersurfaces with free boundary in $\BB^{n+1}$.
\end{abstract}


\section{Introduction}

Isoperimetric inequality is one of the  fundamental topics in differential geometry. The classical isoperimetric inequality in the Euclidean space says that among all bounded domains in $\RR^{n+1}$ with fixed enclosed volume, the minimum of the area functional is achieved by round balls.
The higher order generalizations of isoperimetric inequality in convex geometry are the Alexandrov-Fenchel inequalities for quermassintegrals, which say that for closed convex hypersurfaces $\S\subset \RR^{n+1}$, it holds that
$$\int_\S H_kdA\ge \omega_n^{\frac{k-l}{n-l}}\left(\int_\S H_ldA\right)^{\frac{n-k}{n-l}},$$
with equality holding on the round spheres.  
Here $H_k$ is the normalized $k$-th mean curvature of $\S$ and $\omega_n$ denotes the surface area of $\SS^n$. See \cite{BZ,San,Sch} for instance.

Fix a domain $B$ in $\RR^{n+1}$, whose boundary $\p B$ is named support hypersurface, the partitioning problem or the relative isoperimetric problem is to find hypersurfaces $\S$ in $B$ with least area functional among all hypersurfaces with divide $B$ into two components with prescribed ratio. When $B$ is $\BB^{n+1}$, the unit ball in $\RR^{n+1}$, the solution to this problem is given by all spherical caps or totally geodesic balls with free boundary, cf. \cite{BS, BM, WX1}. Here a hypersurface with free boundary means that it  intersects with $\p \BB^{n+1}=\SS^n$ orthogonally. 
The higher order generalization of relative isoperimetric inequality --  the relative Alexandrov-Fenchel inequalities in $\BB^{n+1}$ have been considered by Scheuer-Wang-Xia in \cite{SWX}. They gave the definition of quermassintegral in $\BB^{n+1}$ from the viewpoint of the first variational formula, and they proved the highest order Alexandrov-Fenchel inequalities and the Gauss-Bonnet-Chern formula for convex hypersurfaces with free boundary  in $\BB^{n+1}$. Willmore type geometric inequalities for  convex hypersurface  with free  boundary in a  ball have been obtained in \cite{LS17, SWX}.

Motivated by the study of the equilibrium shapes of a liquid confined in a given container,  there is a more general relative isoperimetric problem in a fixed domain $B$,  where the energy functional 
\begin{eqnarray}\label{energy-func}
|\p\O\cap B|-\cos \theta |\p B\cap \O|
\end{eqnarray}
 is involved. Here $\O\subset B$ and $\theta\in (0, \pi)$, $|\p B\cap \O|$ is the so-called wetting energy. For more physical background, we refer to Finn's book \cite{Finn}.
It is known that among all hypersurfaces which divide $\BB^{n+1}$ into two components with prescribed energy, the ones with least energy functional are all spherical caps or totally geodesic balls which intersects at the constant angle $\theta$, cf. \cite{WX1}. We shall call such boundary condition the $\theta$-capillary boundary.

In the same spirit of \cite{SWX}, we are interested in finding the quermassintegral for convex hypersurfaces with  $\theta$-capillary boundary in $\BB^{n+1}$, which should be the right higher order generalization of the energy functional. Moreover, in this paper, we shall study the Alexandrov-Fenchel inequalities and the Gauss-Bonnet-Chern formula for such hypersurfaces.

\medskip

Let $n\ge 2$. Let $\mathbb{B}^n$ and $\mathbb{B}^{n+1}$ be  the $n$-dimensional and $(n+1)$-dimensional open Euclidean unit ball centered at the origin, respectively. Let $\SS^n=\p \mathbb{B}^{n+1}$ and $\bar N$ denote the unit outward normal of $\SS^n$.
Let $\S\subset \overline{\mathbb{B}}^{n+1}$ be a smooth, orientable, compact hypersurface with boundary, given by an isometric embedding 
$x:\overline{\mathbb{B}}^n\to \overline{\mathbb{B}}^{n+1}$.  $\S$ is said to be properly embedded, if $$\Sigma=x(\mathbb{B}^n)\subset {\mathbb{B}}^{n+1},\quad \partial \Sigma=x(\p \mathbb{B}^n)\subset \SS^n.$$
We are concerned with the case when $\S$ is a convex hypersurface with capillary boundary. We choose $\nu$ to be the unit normal of $\S$ such that $\S$ is convex in the sense that the second fundamental form defined by $h(X, Y)= -\<D_X Y, \nu\>$ is non-negative definite, where $D$ is the Euclidean connection. 
\begin{defn}\label{capillary bdry}
	  For $\theta\in (0,\pi)$, $\Sigma$ is said to be with a $\theta$-capillary boundary, if $\S$ intersects $\SS^n$ at the constant angle $\theta$, that is, 	
		\begin{eqnarray*}
\langle \nu,\overline{N}\circ x\rangle=-\cos\theta, \hbox{ along }\p \mathbb{B}^n.
\end{eqnarray*}
 In particular, if $\theta=\frac{\pi}{2}$, or equivalently say that $\Sigma$  intersects $\partial \mathbb{B}^{n+1}$ orthogonally, we say that $\Sigma$ is with free boundary.	
\end{defn}
The model examples of properly embedded hypersurfaces with a $\theta$-capillary boundary  are {\it the spherical cap of radius $r\in (0,\infty)$ around some constant unit vector field $e$ in $\RR^{n+1} $with $\theta$-capillary boundary}, given by
\begin{eqnarray}\label{static model}
C_{ \theta,r}(e):=\{x\in \overline{\BB}^{n+1}: \left|x- \sqrt{r^2+2r\cos\theta+1} e\right|= r\},
\end{eqnarray}
and {\it the flat ball   around $e\in \SS^n$ with $\theta$-capillary boundary}, given by
\begin{eqnarray}\label{static model0}
C_{ \theta,\infty}(e):=\{x\in \overline{\BB}^{n+1}: \<x, e\>=\cos\theta\}.
\end{eqnarray}
We shall drop the argument $e$, in cases where it is not relevant.

For our purpose, we always assume $\S$ is a convex hypersurface with $\theta$-capillary boundary for $\theta\in (0,\frac{\pi}{2}]$.
Then $\p\S\subset \SS^n$ is a strictly convex hypersurface in $\SS^n$ (see Proposition \ref{basic-capillary}) and it bounds a strictly convex body in $\SS^n$, which we denote by $\widehat{\p\S}$,  cf. \cite{dCW, GX}.  Denote by $\widehat{\S}$ the domain enclosed by $\S$ in $\mathbb{B}^{n+1}$ which contains $\widehat{\p\S}$.
Let $\s_k$ denote the $k$-th elementary symmetric polynomial, evaluated at the principal curvatures of $\S$, and $$H_k=\frac{1}{\binom{n}{k}}\s_k, \quad 1\le k\le n.$$We make the convention that $$\s_0=H_0=1 \quad \s_{n+1}=H_{n+1}=0.$$
Now we define the following geometric functionals for convex hypersurfaces with $\theta$-capillary boundary in $\mathbb{B}^{n+1}$ as follows, which we expect to be the correct counterparts to the quermassintegrals for closed convex hypersurfaces in $\RR^{n+1}$:
 \begin{eqnarray}\label{relative quer 1}
&&	 W_{0,\theta}(\widehat{\Sigma}):= |\widehat{ \S}|,\notag
\\&&
	 W_{1,\theta}(\widehat{\Sigma}):=\frac{1}{n+1}\left(|\Sigma|-\cos\theta W_0^{\mathbb{S}^n}(\widehat{\p\S})\right),
\end{eqnarray}
\begin{eqnarray}\label{relative quer n}
 W_{n+1,\theta}(\widehat{\S}):=\frac{1}{n+1} \left[ \int_\S H_n dA -  \sum_{l=0}^{n-1}(-1)^{n+l}\binom{n}{l} \cos^{n-1-l}\theta \sin^l\theta  W_l^{\SS^n}(\widehat{\p\S})\right],
\end{eqnarray}
and for $1\leq k\leq n-1$,
\begin{eqnarray}\label{relative quer k}
		&&	 W_{k+1,\theta}(\widehat{\S}): = \frac{1}{n+1} \Bigg\{ \int_\S H_k dA- \cos\theta \sin^{k }\theta   W_k^{\SS^{n}}(\widehat{\p\S}) \\&&\quad -\sum_{l=0}^{k-1} \frac{(-1)^{k+l}}{n-l}\binom{k}{l}\Big[(n-k) \cos^2  \theta   +   (k-l)  \Big]  \cos^{k-1-l}\theta\sin^l\theta  W_l^{\SS^n}(\widehat{\p\S})\Bigg\}.\notag
\end{eqnarray} Here for a $k$-dimensional submanifold $M\subset \RR^{n+1}$ (with or without boundary), $|M|$ always denotes the $k$-dimensional Hausdorff measure of $M$. $ W_{k}^{\mathbb{S}^n}$ denotes the $k$-th quermassintegral of the closed convex hypersurface $\p\S\subset \mathbb{S}^n$, see Section \ref{sec 2.3} for more discussion.
We remark that  $W_{1,\theta}(\widehat{\S})$ is just a multiple of the energy functional \eqref{energy-func} we mentioned before, or see e.g. \cite{RR,WW,WX1}.
When $\theta=\frac{\pi}{2}$, $W_{k, \theta}(\widehat{\S})$ is exactly the quermassintegrals defined in \cite{SWX}.

 The reason to define $ W_{k+1,\theta}$ to be the quermassintegral for convex hypersurfaces with $\theta$-capillary boundary in $\overline{\mathbb{B}}^{n+1}$ is the following first variational formula.
  \begin{thm}\label{thm 0}
 	Let $\S_t\subset \overline{\BB}^{n+1}$ be a family of smooth, properly embedded hypersurfaces with $\theta$-capillary boundary, given by $x(\cdot, t): \overline{\BB}^{n}\to \overline{\BB}^{n+1}$, such that $$(\p_t x)^\perp=f\nu,$$ for some normal speed function $f$.  Then for $1\leq k\leq n$,
 	\begin{eqnarray}\label{fvf0}
 		\frac{d}{dt} W_{k,\theta}(\widehat{\Sigma_t})=\frac{n+1-k}{n+1}\int_{\S_t} H_k f dA_t,
 	\end{eqnarray} and 
 \begin{eqnarray*}
	\frac{d}{dt} W_{n+1,\theta}(\widehat{\Sigma_t})=0.
 \end{eqnarray*}
 \end{thm}

 We obtain the following Alexandrov-Fenchel type inequalities and Gauss-Bonnet-Chern formula for convex hypersurfaces with $\theta$-capillary boundary in $\overline{\mathbb{B}}^{n+1}$.
\begin{thm}\label{thm 1}
Let  $\theta \in (0,\frac{\pi}{2}]$ and $\Sigma\subset \overline{\mathbb{B}}^{n+1}$ be a convex hypersurface with $\theta$-capillary boundary.
Then \begin{eqnarray}\label{gbc}
		 W_{n+1,\theta}(\widehat{\S})=\frac{\omega_n}{2(n+1)}I_{\sin^2\theta}(\frac{n}{2},\frac{1}{2}),
	\end{eqnarray}
	where $I_s(\frac{n}{2},\frac{1}{2}) $ is the   regularized incomplete beta function given by $$I_s(\frac{n}{2},\frac{1}{2})=\frac{\int_0^s t^{\frac{n}{2}-1}(1-t)^{-\frac{1}{2}}dt}{\int_0^1 t^{\frac{n}{2}-1}(1-t)^{-\frac{1}{2}}dt}.$$
We also have that for $0\leq k\leq n-1$,
	\begin{eqnarray}\label{af ineq}
		 W_{n,\theta}(\widehat{\S})\geq (f_n\circ f_k^{-1})( W_{k,\theta}(\widehat{\S})),
	\end{eqnarray}where $f_k:=f_k(r)$ is the strictly increasing real function given by
	\begin{eqnarray*}
		f_k(r):= W_{k,\theta}(\widehat{C_{ \theta,r}}),
	\end{eqnarray*} where $C_{\theta,r}$ is the spherical cap given by \eqref{static model}.  Moreover,  equality holds if and only if $\Sigma$ is a spherical cap or a flat ball with $\theta$-capillary boundary.
\end{thm}
In particular, for $n=2$, we obtain  a Minkowski  type inequality for convex surfaces in $\overline{\mathbb{B}}^3$ with $\theta$-capillary boundary.
\begin{cor}\label{cor 1}	Let $\theta \in (0,\frac{\pi}{2}]$ and $\Sigma\subset \overline{\mathbb{B}}^3$ be a convex surface with $\theta$-capillary boundary. Then
\begin{eqnarray*}
3W_{3, \theta}(\widehat{\S})= \int_\S H_2dA-\cos\theta|\widehat{\p\S}|+\sin\theta|\p\S|  = {2\pi} (1- \cos\theta),
\end{eqnarray*}
and
	\begin{eqnarray}\label{minkowski ineq}
  W_{2, \theta}(\widehat{\S})&=&\frac{1}{6}\left(\int_\S H dA -\sin\theta\cos\theta|\p\S|+(1+ \cos^2\theta)|\widehat{\p\S}|\right)\nonumber \\
  &\geq & (f_2\circ f_1^{-1})\left(\frac{1}{3}(|\Sigma|-\cos\theta |\widehat{\p\S}|)\right).
	\end{eqnarray}
Moreover,  equality holds if and only if $\Sigma$ is a spherical cap or a flat disk with $\theta$-capillary boundary.
\end{cor}
We remark that, when $\theta=\frac{\pi}{2}$, Theorems \ref{thm 0} and \ref{thm 1} have been proved in \cite{SWX}.
 
The method of our proof follows the classical method for proving geometric inequalities by employing monotonicity properties along and convergence of a suitable curvature flow. We will use the locally constrained inverse type curvature flow with capillary boundary in this paper to achieve the inequality \eqref{af ineq} in Theorem \ref{thm 1}. Such locally constrained flow has been first considered by Guan-Li \cite{GL15}, and used in \cite{BGL, CGLS, GL15, GL18, GLW,  HLW} for closed hypersurfaces in space forms, and in \cite{SWX, WX2} for hypersurfaces with free boundary in balls to prove various geometric inequalities. The Minkowski formulas play an essential role to design these flows. We refer to \cite{BGL, SWX, SX} for more description to such locally constrained flows.
We also refer to \cite{ACW, AW18, ChS, DF, GWW, HL, LWX, MS, WX14}  and references therein for the studying of the Alexandrov-Fenchel type inequalities for closed hypersurfaces in space forms which used various types of curvature flow. 
In addition to the curvature flow approach, there are also many other interesting methods to study the Alexandrov-Fenchel inequalities for closed hypersurfaces, see e.g.  \cite{ CW13,CW14, Qiu, QX19, Tru}. Besides,  we could  establish the Alexandrov-Fenchel inequalities for capillary hypersurfaces in the half-space in a forthcoming paper \cite{WWX21}.

For our purpose, we first generalize the Minkowski formula for hypersufaces with capillary boundary in \cite{WX1, WW} to higher order mean curvatures,
\begin{eqnarray}\label{minkowski sigma_k0}  
			\int_\S  H_{k-1}  \left( \langle x,e\rangle+\cos\theta \langle \nu,e\rangle\right)dA=\int_\S H_k  \langle X_e,\nu\rangle dA,
		\end{eqnarray}
where $e$ is a constant unit vector field in $\RR^{n+1}$ and $X_e$ is defined in \eqref{Xe}, which is a conformal Killing vector field parallel to $\SS^n$.
Then we consider the flow 
 \begin{equation}\label{flow with capillary0}
\left\{ \begin{array}{lll}
&\partial_t x(\cdot, t)  =f(\cdot, t) \nu(\cdot, t)+T(\cdot, t), \quad &
 \hbox{ in }\mathbb{\BB}^n\times[0,T),\\
&\langle \nu(\cdot, t),\overline{N}\circ x(\cdot,t)\rangle  = -\cos\theta 
 \quad & \hbox{ on }\partial \mathbb{\BB}^n\times [0,T),
 \end{array}\right.
 \end{equation}
where $T(\cdot, t)$ is the tangential component of $\p_t x$ and 
$$f=\frac{\langle x+\cos\theta  \nu, e\rangle }{n\sigma_n/\sigma_{n-1}}-\langle X_e,\nu\rangle.$$
By the first variational formula \eqref{fvf0} and \eqref{minkowski sigma_k0}, it is easy to see that
under flow \eqref{flow with capillary0}, $W_{n,\theta}(\widehat{\S})$ is preserved and $ W_{k,\theta}(\widehat{\S}), 1\le k\le n-1$ is non-decreasing.

The remaining task is to prove the convergence of the flow to a spherical cap with capillary boundary. For this aim, we need to make a choice of $e$, which depends only on the initial hypersurface $x(\overline{\BB}^n, 0)=\S_0$, so that $\<x+\cos\theta \nu, e\>>0$ for $\S_0$, due to the strict convexity and its $\theta-$capillary boundary of $\S_0$. This will be done in Proposition \ref{ellipticity}. Under this choice, since we have barriers for the flow hypersurfaces, given by spherical caps around $e$ with $\theta$-capillary boundary, we are able to prove that under the flow, $\<x+\cos\theta \nu, e\>$ is always bounded below by a uniform constant. On the other hand, we can prove the quantity $\<X_e, \nu\>$ is bounded below by a uniform constant, which enable us to write the flow equation to be a scalar equation of graph functions by using M\"obius transformations, cf. \cite{LS, WW, WX2}. Moreover, we prove that the flow preserves the strict convexity and we have the uniform curvature estimates, which implies the long-time existence. Finally, by using the monotonicity of $W_{k,\theta}(\widehat{\S})$ along the flow,  we obtain the convergence of the flow to a spherical cap around $e$.

\vspace{.2cm}
\textbf{This article is structured as follows.} In Section \ref{sect 2}, we give some preliminaries for $\theta$-capillary hypersurfaces, and prove a new Minkowksi type formula \eqref{minkowski sigma_k}. Besides, we introduce the definition of  quermassintegral for hypersurfaces with $\theta$-capillary boundary and derive its first variational formula, i.e. Theorem \ref{thm 0}. In Section \ref{sect 3}, we introduce a locally constrained type curvature flow \eqref{flow with capillary}, derive the evolution equations for various geometric quantities,  and obtain the uniform curvature estimates for convex hypersurfaces with $\theta$-capillary boundary in a ball along this flow, which  follows the long-time existence and convergence of our flow \eqref{flow with capillary}. The last part is devoted to prove the Alexandrov-Fenchel inequalities  for convex hypersurfaces with $\theta$-capillary boundary, i.e. Theorem \ref{thm 1}.

\

\section{Convex hypersurfaces with $\theta$-capillary boundary}\label{sect 2}

\subsection{Notation and preliminaries}\label{sec 2.1}

Let $\S \subset \RR^{n+1}$ be a smooth embedded, orientable hypersurface. 
Denote $D$ be the Levi-Civita connection of $\mathbb{R}^{n+1}$ with respect to the Euclidean metric $\delta$, and $\n$ be the Levi-Civita connection on $\S$ with respect to induced metric $g$ from the embedding $x$. We denote $\div, \Delta,\n^2$ be the divergence, Laplacian, Hessian operator on $\S$ respectively.
The second fundamental form of $x$  is given by the Gaussian formula
$$D_{X}Y=\n_{X}Y- h(X,Y)\nu.$$
The Weingarten operator is defined via
$$g(\W(X),Y)=h(X,Y),$$
and the Weingarten equation says that
\begin{eqnarray*}
\bar{\n}_{X}\nu=\W(X).
\end{eqnarray*}
The Gauss-Codazzi equation says that
\begin{eqnarray*}
&&\Rm(X,Y,Z,W)=h(Y,Z)h(X,W)-h(Y, W)h(X,Z),\\
&&(\n_Z h)(X, Y)=(\n_Y h)(X, Z).
\end{eqnarray*}
 where our convention for the Riemannian curvature tensor $\Rm$ is
$$\Rm(X,Y,Z,W)= g(\Rm(X,Y)Z, W)=g(\n_{X}\n_{Y}Z-\n_{Y}\n_{X}Z-\n_{[X,Y]}Z, W).$$
\begin{rem}\label{Rem-coordinates}
We will simplify the notation by using the following shortcuts occasionally:
\begin{enumerate}
\item When dealing with complicated evolution equations of tensors, we will use a local frame to express tensors with the help of their components, i.e. for a tensor field $T\in \mathcal{T}^{k,l}(\S)$, the expression $T^{i_{1}\dots i_{k}}_{j_{1}\dots j_{l}}$ denotes  
\begin{eqnarray*}
T^{i_{1}\dots i_{k}}_{j_{1}\dots j_{l}}=T(e_{j_{1}},\dots,e_{j_{l}},\e^{i_{1}},\dots \e^{i_{k}}),
\end{eqnarray*}
where $(e_{i})$ is a local frame and $(\e^{i})$ its dual coframe. 
\item The coordinate expression for  the $m$-th covariant derivative of a $(k,l)$-tensor field $T$ is 
\begin{eqnarray*}
\n^m T=\left(T^{i_1\dots i_k}_{j_1\dots j_l; j_{l+1}\dots j_{l+m}}\right),
\end{eqnarray*}
where indices appearing after the semi-colon denote the covariant derivatives.
\item We shall use the convention of Einstein summation. For convenience the components of the Weingarten map $\W$ are denoted by $(h^{i}_{j})=(g^{ik}h_{kj})$, and $|h|^2$ be the norm square of the second fundamental form, that is $|h|^2=g^{ik}h_{kl}h_{ij}g^{jl}$, where  $(g^{ij})$ is the inverse of $(g_{ij})$. We use the metric tensor $(g_{ij})$ and its inverse $(g^{ij})$ to  lower down and raise up the indices of tensor fields on $\S$.

\end{enumerate}
\end{rem}

Let $\s_k(\k), 1\leq k\leq n,$ be the $k$-th elementary symmetric polynomial for $\k\in \RR^n$ and $H_k(\k)$ be its normalization $H_k(\k)=\frac{1}{\binom{n}{k}}\s_k(\k)$. Denote by $\sigma_k(\k|i)$ the symmetric function $\sigma_k(\k)$ with $\kappa_i=0$. We shall use the following basic properties in this paper.
\begin{prop}\label{prop2.2}\
	\begin{enumerate}
\item $ \sigma_k(\k)=\sigma_k(\k|i)+\k_i\sigma_{k-1}(\k|i),  \quad  \forall 1\leq i \leq n.$
\item $ \sum\limits_{i = 1}^n {\sigma_{k}(\k|i)}=(n-k)\sigma_k(\k).$
\item $\sum\limits_{i = 1}^n {\k_i \sigma_{k-1}(\k|i)}=k\sigma_k(\k)$.
\item $\sum\limits_{i = 1}^n {\k_i^2 \sigma_{k-1}(\k|i)}=\s_1(\k)\sigma_k(\k)-(k+1)\s_{k+1}(\k)$.
	\end{enumerate}
\end{prop}
\begin{prop}\label{prop2.3}
For $1\le k<l\le n$, for  $\kappa\in \Gamma_+:=\{\kappa\in\mathbb{R}^n:\kappa_i>0, 1\leq i\leq n\}$, we have
\begin{eqnarray}\label{NM-ineq}
H_kH_{l}\ge H_{k-1}H_{l+1},
\end{eqnarray}	
with equality holding if and only if $\kappa=c(1, \cdots, 1)$ for some $c>0$.
\end{prop}
\begin{prop}\label{prop2.4}
For $1\le k\le n$, $F(\k)=\frac{\s_k}{\s_{k-1}}(\k)$ is concave in $\Gamma_+$.
\end{prop}
 For a proof of Propositions \ref{prop2.2} -\ref{prop2.4}, one can refer to \cite[Chapter XV, Section 4]{Lie} and \cite[Lemma 2.10, Theorem 2.11]{Spruck} respectively.

\subsection{Basic properties for $\theta$-capillary hypersurfaces}\

Let $\S\subset \overline{\BB}^{n+1}$ be a smooth, properly embedded, orientable hypersurfaces with $\theta$-capillary boundary, given by the embedding $x: \overline{\BB}^{n}\to \overline{\BB}^{n+1}$.
Let $\mu$ be the unit outward co-normal of $\p \S$ in $\S$ and $\overline{\nu}$ be the unit normal to $\partial\Sigma$ in $\SS^n$ such that $\{\nu,\mu\}$ and $\{\overline{\nu},\overline{N}\circ x\}$ have the same orientation in the normal bundle of $\partial\Sigma\subset\overline{\mathbb{B}}^{n+1}$. From Definition \ref{capillary bdry}, it follows that
\begin{equation}\label{co-normal bundle}
	\begin{cases}
		\overline{N}\circ x=\sin\theta \mu-\cos\theta \nu,
		\\
		\overline{\nu}=\cos\theta \mu+\sin\theta \nu.
	\end{cases}
\end{equation}
$\p\S$ can be viewed as a smooth closed hypersurface in $\SS^n$, which bounds $\widehat{\p\S}$ inside $\SS^n$. By our convention, $\bar \nu$ is the unit outward normal of $\p \S$ in $\widehat{ \p\S}\subset \SS^n$. The second fundamental form of $\p \S$ in  $\SS^n$ is given by $$\widehat{h}(X, Y):= -\<\nabla^\SS_X Y, \bar \nu\>= -\<D_X Y, \bar \nu\>, \quad X, Y\in T(\p\S).$$
The second equality holds since $\<\bar \nu, \bar N\circ x\>=0$.
The second fundamental form of $\p \S$ in $\S$ is given by
$$\tilde{h}(X, Y):= -\<\nabla_X Y, \mu\>= -\<D_X Y, \mu\>, \quad X, Y\in T(\p\S).$$
The second equality holds since $\<\nu, \mu\>=0$.
\begin{prop}\label{basic-capillary} Let $\S\subset \overline{\BB}^{n+1}$ be with  $\theta$-capillary boundary. Let $\{e_\a\}_{\a=1}^{n-1}$ be an orthonormal frame of $\p \S$. Then along $\p\S$,
\begin{itemize}
\item[(1)] $\mu$ is a principal direction of $\S$, that is, $h(\mu, e_\a)=0$.
\item[(2)] $h_{\a\b}=\sin\theta \widehat{h}_{\a\b}-\cos\theta \delta_{\a\b}.$
\item[(3)] $\tilde{h}_{\a\b}=\cos\theta \widehat{h}_{\a\b}+\sin\theta \d_{\a\b}= \frac{1}{\sin\theta}\delta_{\a\b}+\cot\theta h_{\a\b}.$
\item[(4)] $h_{\a\b; \mu}=\tilde{h}_{\b\g}(h_{\mu\mu}\delta_{\a\g}-h_{\a\g})$
\end{itemize}
\end{prop}
\begin{proof}The first assertion is well-known, cf. \cite[Proposition 2.1]{WX1}.
For (2), we see 
	\begin{eqnarray*}
		h_{\alpha\beta}&=-\langle D_{e_\alpha}e_\beta,\nu\rangle	=\langle \widehat{h}_{\alpha\beta}\overline{\nu}+\delta_{\alpha\beta}\overline{N},\nu\rangle
	=\sin\theta \widehat{h}_{\alpha\beta}-\cos\theta \delta_{\alpha\beta}.
	\end{eqnarray*}
For (3), we see
\begin{eqnarray*}
\tilde{h}_{\a\b}= -\<D_{e_\alpha}e_\beta, \mu\>= \<\widehat{h}_{\alpha\beta}\overline{\nu}+\delta_{\alpha\beta}\overline{N},\mu\>=\cos\theta \widehat{h}_{\a\b}+\sin\theta \d_{\a\b}.
\end{eqnarray*}

For (4), by taking derivative of $h(\mu, e_\a)=0$ with respect to $e_\b$, using the Codazzi equation and the fact (1),  we get
\begin{eqnarray*}
0=e_\b(h(\mu, e_\a))&=&h_{\a \mu; \b}+h(\n_{e_\b}e_\a, \mu)+h(\n_{ e_\b}\mu, e_\a)
\\&=&h_{\a \b; \mu}+ \<\n_{e_\b}e_\a, \mu\>h_{\mu\mu}+\<\n_{e_\b}\mu, e_\g \>h_{\a\g}
\\&=&h_{\a \b; \mu}-\tilde{h}_{\b\g}(h_{\mu\mu}\delta_{\a\g}-h_{\a\g}).
\end{eqnarray*}
\end{proof}
\begin{cor}\label{convexity}
Let $\S\subset \overline{\BB}^{n+1}$ be with  $\theta$-capillary boundary for $\theta\in (0, \frac{\pi}{2}]$. If $\S$ is convex (strictly convex resp.), in the sense $h$ is  nonnegative definite (positive definite resp.), then $\p \S$ is convex (strictly convex resp.) in both $\SS^n$ and $\S$.
\end{cor}
\begin{proof} It follows directly from Proposition \ref{basic-capillary} (2) and (3) that $\widehat{h}$ and $\tilde{h}$ are nonnegative definite (positive definite resp.) provided $h$ is so.\end{proof}

\subsection{Minkowski formula}\

For a constant unit vector field $e\in \RR^{n+1}$, denote \begin{eqnarray}\label{Xe}
X_e:=\langle x,e\rangle x-\frac{1}{2}(|x|^2+1)e.
\end{eqnarray}

It has been observed in \cite[Proposition 3.1]{WX1} that $X_e$ is a conformal Killing vector field such that
\begin{eqnarray}\label{conformal-K}
\mathcal{L}_{X_e}\delta= \<x, e\>\delta,
\end{eqnarray}
and $$\<X_e, \bar N\>=0 \hbox{ along }\SS^n.$$
By using such property of $X_e$, the following Minkowski type formula has been proved in \cite[Proposition 3.2]{WX1}: for a properly embedded hypersurface in $\overline{\BB}^{n+1}$ with $\theta$-capillary boundary, 
\begin{eqnarray}\label{mink-1}
	n\int_\S \big(\langle x,e\rangle+  \cos\theta \langle \nu,e\rangle \big)dA =\int_\S H\langle X_e,\nu\rangle dA.
\end{eqnarray}
\begin{rem}
From a direct calculation, we see that $C_{\theta,r}(e)$ satisfies
\begin{eqnarray}\label{static-eq}
	\langle x,e\rangle+  \cos\theta \langle \nu,e\rangle =\frac{1}{r}\langle X_e,\nu\rangle.
\end{eqnarray}
\end{rem}
For our purpose, we generalize \eqref{mink-1} to higher order Minkowski type formulas.
\begin{prop}\label{minkowski formulae 1}Let $\S\subset \overline{\BB}^{n+1}$ be a smooth, properly embedded, orientable hypersurfaces with $\theta$-capillary boundary, given by the embedding $x: \overline{\BB}^{n}\to \overline{\BB}^{n+1}$.
For $1\leq k\leq n$, it holds
		\begin{eqnarray}\label{minkowski sigma_k}
			\int_\S  H_{k-1}  \left( \langle x,e\rangle+\cos\theta \langle \nu,e\rangle\right)dA=\int_\S H_k  \langle X_e,\nu\rangle dA.
		\end{eqnarray}
\end{prop}
\begin{proof}Let $\{e_i\}_{i=1}^n$ be an orthonormal frame on $\S$.
Define $$P_e:=\langle \nu,e\rangle x-\langle x,\nu\rangle e.$$ Let $X_e^T$ and $P_e^T$ be the tangential projection of $X_e$ and $P_e$ on $\Sigma$ respectively.
From \eqref{conformal-K}, we have
\begin{eqnarray}\label{x a}
\frac12(\n_{i}(X_e^T)_j+\n_{j}(X_e^T)_i)=  \<x, e\>g_{ij}-h_{ij}\langle X_e,\nu\rangle.
\end{eqnarray}
On the other hand, since $\<P_e, \nu\>=0$, by a direct computation, we know 
	\begin{eqnarray}\label{p a}
	 \nabla_{i} (P_e^T)_{j} =D_{i} (P_e)_{j}=\langle \nu,e \rangle g_{ij}+h_{il}\langle e,e_l\rangle \langle x,e_j\rangle-h_{il}\langle x,e_l\rangle \langle e,e_j\rangle . 
\end{eqnarray}
By multiplying $\sigma_{k-1}^{ij} $ into above equations \eqref{x a} and \eqref{p a}, and combining with  Proposition \ref{prop2.2}, it follows that
	\begin{eqnarray*}
 &&\sigma^{ij}_{k-1}\n_i (X_e^T+\cos\theta P_e^T)_j 
=(n-k+1)\sigma_{k-1}\left( \langle x,e\rangle+\cos\theta \langle \nu,e\rangle\right)  -k\sigma_k  \langle X_e,\nu\rangle.
\end{eqnarray*}
We need to show
\begin{eqnarray*}
 &&\int_\S \sigma^{ij}_{k-1}\n_i (X_e^T+\cos\theta P_e^T)_j dA=0,
\end{eqnarray*}
In fact, by integration by parts, and using the fact that $\n_i\sigma^{ij}_{k-1}=0$ and $\mu$ is a principal direction of $\S$ along   $\p \S$, we get
\begin{eqnarray*}
 &&\int_\S \sigma^{ij}_{k-1}\n_i (X_e^T+\cos\theta P_e^T)_j dA= \int_{\p \S} \sigma_{k-1}^{\mu\mu} \<X_e+\cos\theta P_e,\mu\> ds.
\end{eqnarray*}
From \eqref{co-normal bundle}, on $\p\S$, we see that 
\begin{eqnarray*}
\<X_e,\mu\>&=&\<X_e,\sin\theta \overline{N}+\cos\theta \overline{\nu}\>=\cos\theta \<X_e,\overline{\nu}\>
\\&=&\cos\theta \left( \<x,e\>\<x,\overline{\nu}\>-\<e,\overline{\nu}\>\right)=-\cos\theta \<e,\overline{\nu}\>,
\end{eqnarray*}and
\begin{eqnarray*}
\<P_e,\mu\>=\<\nu,e\>\<x,\mu\>-\<x,\nu\>\<e,\mu\>=\<\sin \theta \nu+\cos\theta \mu,e\>=\<\overline{\nu},e\>,
\end{eqnarray*}which follows that
$$\<X_e+\cos\theta P_e,\mu\>=0.$$ We complete the proof.
\end{proof}

\subsection{Quermassintegrals and first variational formula}\label{sec 2.3}\

Recall that given a convex body $K\subset \mathbb{S}^n$ with non-empty smooth boundary $\p K$, its quermassintegrals $W^{\mathbb{S}^n}_k$  are defined by
\begin{eqnarray*}
	 W_0^{\mathbb{S}^n}(K):=|K|,\quad \quad   W_1^{\mathbb{S}^n}(K):=  \frac{|\p K|}{n} ,
\end{eqnarray*}
and for $2\leq k\leq n$, 
\begin{eqnarray*}
	 W_k^{\mathbb{S}^n}(K)&:=\frac{1}{n} \int_{\p K} H_{k-1}^{\mathbb{S}^n} dA+\frac{k-1}{n-k+2}  W^{\mathbb{S}^n}_{k-2}(K),
\end{eqnarray*}where 
$H_k^{\mathbb{S}^n}=\frac{1}{\binom{n-1}{k}}\sigma_k^{\mathbb{S}^n}$ and $\sigma_k^{\mathbb{S}^n}$ denote the $k$-th elementary symmetric polynomials, which evaluated at the $(n-1)$-principal curvatures of the hypersurface $\p K\subset \mathbb{S}^n$, and we use the convention that $H_0^{\mathbb{S}^n}=1$, $H_n^{\mathbb{S}^n}=0$.
In particular, 
$$W_n^{\mathbb{S}^n}(K)=\frac{\o_{n-1}}{n},$$ due to the spherical Gauss-Bonnet-Chern's Theorem, cf. \cite{Sol}.
The following first variational formula of $W^{\mathbb{S}^n}_k$ is proved by Reilly \cite[Section 3]{Reilly}, see also Barbosa-Colares \cite[Section 4]{BdC}. 
\begin{prop}\label{fvf-sphere}Consider a family of convex bodies  $\{K_t\}$ whose boundary $\p K_t$ evolving by a normal speed function $f$, then for $0\leq k\leq n$, 
\begin{eqnarray}\label{1st var on sphere}
	\frac{d}{dt} W_k^{\mathbb{S}^n}(K_t)= \frac{1}{\binom{n}{k}}\int_{\p K_t} \sigma_k^{\mathbb{S}^n} f dA_t=\frac{n-k}{n} \int_{\p K_t}  H_k^{\mathbb{S}^n} f dA_t.
\end{eqnarray}
\end{prop}

\vspace{.4cm}
Next we define the geometric quantities ${W}_{k,\theta}$ for smooth, properly embedded, convex hypersurfaces in $\overline{\BB}^{n+1}$ with $\theta$-capillary boundary. 
Let $\S\subset \overline{\BB}^{n+1}$ be such a hypersurface.
 Denote $\widehat{\S}$ be the enclosed bounded convex domain by $\Sigma$ inside  $\mathbb{B}^{n+1}$ and  $\widehat{\p\S}\subset \mathbb{S}^n$ be the enclosed convex domain by $\p\S\subset \mathbb{S}^n$. We define the $ W_{k,\theta}$ for $\widehat{\S}\subset \overline{\BB}^{n+1}$ as in \eqref{relative quer 1}, \eqref{relative quer n} and \eqref{relative quer k}. 
 In particular, 
\begin{eqnarray*} 
	 W_{2,\theta}(\widehat{\S})=\frac{1}{n(n+1)}\Big\{ \int_\S H dA- \sin\theta\cos\theta|\p\S|+\big[1+(n-1)\cos^2\theta\big]|\widehat{\p\S}| \Big\}.
\end{eqnarray*}
We have the following first variational formula for $ W_{k,\theta}$.
   \begin{thm}\label{fvf}
 	Let $\S_t\subset \overline{\BB}^{n+1}$ be a family of smooth, properly embedded hypersurfaces with $\theta$-capillary boundary, given by $x(\cdot, t): \overline{\BB}^{n}\to \overline{\BB}^{n+1}$, such that 
	\begin{eqnarray}\label{normal flow}
(\p_t x)^\perp=f\nu, 
\end{eqnarray}
for some normal speed function $f$.  Then for $0\leq k\leq n-1$,
 	\begin{eqnarray}\label{fvf-1}
 		\frac{d}{dt} W_{k+1,\theta}(\widehat{\Sigma_t})=\frac{n-k}{n+1}\int_{\S_t} H_{k+1} f dA_t,
 	\end{eqnarray} and 
 \begin{eqnarray}\label{fvf n+1}
	\frac{d}{dt} W_{n+1,\theta}(\widehat{\Sigma_t})=0.
 \end{eqnarray}
 \end{thm}

Before proving Theorem \ref{fvf}, we study the evolution equations for several geometric quantities under the flow
\begin{eqnarray}\label{flow with normal and tangential}
	\p_t x=f\nu+T,
\end{eqnarray}where $T\in T\Sigma_t$. 
\begin{prop}\label{basic evolution eq}
Along the general flow \eqref{flow with normal and tangential}, it holds that
\begin{enumerate} 
	\item $\p_t g_{ij}=2fh_{ij}+\n_i T_j+\n_jT_i$.
	\item $\p_tdA_t =\left(fH+\div(T)\right)dA_t.$ 
	\item $\p_t\nu =-\n f+h(e_i,T)e_i$.	
	\item $\p_t h_{ij}=-\n^2_{ij}f +fh_{ik}h_{j}^k +\n_T h_{ij}+h_{j}^k\n_iT_k+h_{i}^k\n_j T_k.$
	\item $\p_t h^i_j=-\n^i\n_{j}f -fh_{j}^kh^{i}_k+\n_T h^i_j.$
	\item $\p_t H=-\Delta f-|h|^2 f+ \n_T H$.
	\item  $\p_t F= -F^{j}_i\n^i\n_{j}f -fF^{j}_ih_{j}^kh^{i}_k+\n_TF$, for $F=F(h^j_i)$, where $F^i_j:=\frac{\p F}{\p h^j_i}$.
	\item $\p_t \s_k=- \frac{\p \s_k}{\p h_i^j}\n^i\n_{j}f - f (\s_1\s_k-(k+1)\s_{k+1})+\n_T \s_k$. 
\end{enumerate}
\end{prop}  
 \begin{proof} Due to the tangential part appearance in flow \eqref{flow with normal and tangential}, we include a proof below for completeness. The assertion for $T=0$ can be found for example in  \cite[Chapter 2, Section 2.3]{Ger} or \cite[Appendix B]{Ec}. 
 
 We choose an orthonormal frame $\{e_i\}_{i=1}^n$ around some point $p$. Recall the Gauss-Weingarten formula\begin{eqnarray*}
	D_{e_i}x=e_i, \quad D_{e_i}e_j=\n_{e_i}e_j-h_{ij}\nu,\quad D_{e_i}\nu=h_{ik}e_k.
\end{eqnarray*}
We compute \begin{eqnarray*}\p_t g_{ij}&=&\p_t\langle D_{e_i}x, D_{e_j} x\rangle =	\langle D_{e_i}(f\nu+T), e_j\rangle +	\langle e_i, D_{e_j}(f\nu+T)\rangle	\\&=&
	2fh_{ij}+ \n_i T_j+\n_jT_i.\end{eqnarray*}
It follows that \begin{eqnarray*}
	\p_t dA_t=\frac12g^{ij}\p_tg_{ij}dA_t
=g^{ij}(fh_{ij}+\n_iT_j)dA_t
= (fH+\div(T))dA_t.
\end{eqnarray*}
Since $\langle \p_t\nu,\nu\rangle=0$, then  \begin{eqnarray*}	\p_t\nu =\<\p_t\nu, D_{e_i}x\>e_i= -\langle \nu, \p_tD_{e_i}x\rangle e_i
	=-\n  f+h(e_i,T) e_i.\end{eqnarray*}
We next compute \begin{eqnarray*}
\p_t h_{ij}&=&\p_t\<D_{e_i}x, D_{e_j}\nu\>
=\<D_{e_i}(f\nu+T), h_{jk}e_k\>+\<e_i, D_{e_j}(-\nabla f+h(e_k, T)e_k)\>
\\&=& -\n^2f(e_i, e_j)+fh_{ik}h_{jk}+\n_T h_{ij}+h_{jk}\n_iT_k+h_{ik}\n_jT_k,
\end{eqnarray*}
where in the last equality we have used
\begin{eqnarray*}
  \<e_i, D_{e_j}(h(e_k, T)e_k)\>&=&e_j(h(e_i, T))+h(e_k, T)\<e_i, \n_{e_j}e_k\>
 \\&=&\n_{T}h(e_i, e_j)+h(e_i, \n_{e_j}T)
=\n_T h_{ij}+h_{ik}\n_jT_k.
\end{eqnarray*}
It follows that
 \begin{eqnarray*} \p_t h^i_j &=&\p_t (g^{ik}h_{kj}) =-(2f h_{ik}+\n_i T_k+\n_kT_i)h_{kj}\\&&-\n^2_{ij}f +fh_{ik}h_{kj} +\n_T h_{ij}+h_{jk}\n_iT_k+h_{ik}\n_jT_k\\&=&-\n^2_{ij}f -fh_{ik}h_{kj}+\n_T h_{ij}+h_{ik}\n_jT_k- h_{jk}\n_kT_i.\end{eqnarray*}
The last three assertions (6)-(8) follow directly from (5), just noticing that
$$F^{ij}(h_{ik}\n_jT_k- h_{jk}\n_kT_i)=0,$$
and the fact
$$\frac{\p \s_k}{\p h^i_j}h_{j}^kh^{i}_k = \s_1\s_k-(k+1)\s_{k+1},$$
which follows from Proposition \ref{prop2.2} (4).
\end{proof}

Let $\S_t\subset \overline{\BB}^{n+1}$ be a family of smooth, properly embedded hypersurfaces with $\theta$-capillary boundary, evolving by \eqref{normal flow}. Then the tangential component $(\p_tx)^T$ of $\p_t x$, which we denote by $T\in T\S_t$ must satisfy \begin{eqnarray}\label{choiceT1}
T |_{\p \S_t}=f\cot\theta \mu + \tilde{T},
\end{eqnarray}where $\tilde{T}\in T(\p\S_t)$. In fact, the restriction of $x(\cdot, t)$ on $\p\BB^n$ is contained in $\SS^n$ and thus,
$$f\nu+T|_{\p \S_t}= \p_t x|_{\p\BB^n}\in T\SS^n.$$
By virtue of \eqref{co-normal bundle}, we have $$\nu=\frac{1}{\sin \theta} \bar \nu -\cot\theta \mu.$$ Since $\bar \nu\in T\SS^n$, we see $(T-f\cot\theta \mu) \in  T\SS^n\cap T\S_t$. Then \eqref{choiceT1} follows. 
Up to a diffeomorphism of $\p \BB^n$, we can assume $\tilde{T}=0$. 
For simplicity, in the following, we always assume that 
\begin{eqnarray}\label{choiceT}
T |_{\p \S_t}=f\cot\theta \mu.
\end{eqnarray}
 \begin{prop}\label{robin bdry of speed1}
Let $\S_t\subset \overline{\BB}^{n+1}$ be a family of smooth, properly embedded hypersurfaces with $\theta$-capillary boundary. Then 
 \begin{eqnarray}\label{robin bdry of speed}
		\n_\mu f=\left(\frac{1}{\sin\theta} +\cot\theta h_{\mu\mu}\right)f \text{ along }\p \S_t.
	\end{eqnarray}
 \end{prop}
 \begin{proof}
 Let $\{e_\alpha\}_{\alpha=1}^{n-1}$ be  an orthonormal frame  of $T(\p\S_t)\subset T\mathbb{S}^n$. Then $\{(e_\alpha)_{\alpha=1}^{n-1}, \mu\}$ forms an orthonormal frame for $T\S_t$. 
Since $$\langle \nu,\overline{N}\circ x\rangle=-\cos\theta, \hbox{ along }\p\S_t,$$ by taking the time derivative and  using Proposition \ref{basic evolution eq} (3), \eqref{co-normal bundle} and \eqref{choiceT}, we obtain along $\p\S_t$,
	\begin{eqnarray*}
		0&=&\langle \p_t \nu, \overline{N} \rangle +\langle \nu, D_{f\nu+T}\overline{N}\rangle
		\\&=&\langle -\n  f+h(e_i,T)e_i, \sin\theta \mu- \cos\theta \nu\rangle +   \langle \nu, f\nu+T\rangle  
		\\&=&-\sin\theta \n_\mu f+\cos\theta h(\mu,\mu) f+ f.
	\end{eqnarray*} The assertion follows. \end{proof}

\noindent{\it Proof  of Theorem \ref{fvf}.}
By Proposition \ref{basic evolution eq}, using integration by parts and the fact $\mu$ is a principal direction, we obtain
\begin{eqnarray*}
	\frac{d}{dt} \int_{\Sigma_t} \sigma_k dA_t &=&\int_{\Sigma_t} \big[-\frac{\p \s_k}{\p h_i^j}\n^i\n_{j}f - f (\s_1\s_k-(k+1)\s_{k+1})+\n_T \s_k \big] dA_t \\&&+\int_{\Sigma_t} \sigma_k\left(f\sigma_1+\div(T)\right)dA_t
	\\&=&(k+1)\int_{\Sigma_t} f\sigma_{k+1} dA_t+ \int_{\p\S_t} \big(\sigma_k \langle T,\mu\rangle - \s_k^{\mu\mu}\n_\mu f\big).
	\end{eqnarray*}
Using \eqref{choiceT}, \eqref{robin bdry of speed} and Proposition \ref{prop2.2} (1), we see along $\p\S_t$,
	\begin{eqnarray*}
 \sigma_k \langle T,\mu\rangle - \s_k^{\mu\mu}\n_\mu f&=& f\left[\cot\theta ( \sigma_k-\s_k^{\mu\mu}h_{\mu\mu})-\frac{1}{\sin\theta} \s_k^{\mu\mu}\right]\\&=&f \left[\cot\theta  \sigma_k(h|h_{\mu\mu})-\frac{1}{\sin\theta}\sigma_{k-1}(h|h_{\mu\mu})\right].
\end{eqnarray*}
Recall in Proposition \ref{basic-capillary} (2), we have $$h_{\a\b}=\sin \theta \widehat{h}_{\a\b}-\cos \theta \delta_{\a\b},$$ for an orthonormal frame  $\{e_\alpha\}_{\alpha=1}^{n-1}$ of $T(\p\S_t)$,
Thus \begin{eqnarray*}
\sigma_k(h|h_{\mu\mu})&=&\sigma_k(\sin \theta \widehat{h}-\cos \theta I_{n-1}),
\end{eqnarray*}
where $I_{n-1}$ is the $(n-1)$ identity matrix. 
In general, for a $(n-1)\times(n-1)$ symmetric matrix $B$ and $0\leq k\leq n-1$, we know
\begin{eqnarray*}
	\s_k(I+B)&=&\sum_{l=0}^k \binom{k}{l} \s_k({\underbrace{I, \cdots, I}_{k-l}, \underbrace{B, \cdots, B}_{l}})
	\\&=&\sum_{l=0}^k \binom{k}{l} \frac{\binom{n-1}{k}}{\binom{n-1}{l}}\s_l(B)=\sum_{l=0}^k \binom{n-l-1}{n-k-1} \s_l(B),
\end{eqnarray*}
Thus we have 
\begin{eqnarray}
\sigma_k(h|h_{\mu\mu})&=&\sigma_k(\sin \theta \widehat{h}-\cos \theta I_{n-1}) \notag
\\&=&(-\cos \theta)^k\sum_{l=0}^k \binom{n-l-1}{n-k-1}(-\tan\theta)^l\s_l(\widehat{h}).\label{sigma alge}
\end{eqnarray}
It follows that
\begin{eqnarray*}
&&\cot\theta  \sigma_k(h|h_{\mu\mu})-\frac{1}{\sin\theta}\sigma_{k-1}(h|h_{\mu\mu})
\\&=& \cos\theta \sin^{k-1}\theta\s_k(\widehat{h})
\\&&+\frac{\cos^{k-1} \theta}{\sin\theta}  \sum_{l=0}^{k-1} (-1)^{k+l} \left[\cos^2 \theta\binom{n-l-1}{n-k-1}+\binom{n-l-1}{n-k}\right]\tan^l\theta\s_l(\widehat{h}).
\end{eqnarray*}
Recall that $\widehat{h}$ is the second fundamental form of $\p \S$ as a hypersurface in $\SS^n$, we have $\s_k(\widehat{h})=\s_k^{\SS^n}$.
It follows that
\begin{eqnarray*}
	\frac{d}{dt}\int_{\S_t} \sigma_kdA_t &=& (k+1)\int_{\S_t} f\sigma_{k+1} dA_t+ \cos\theta \sin^{k-1}\theta \int_{\p\S_t} f \sigma_k^{\mathbb{S}^n}
	\\&&+\frac{\cos^{k-1} \theta}{\sin\theta}  \sum_{l=0}^{k-1} (-1)^{k+l} \left[\cos^2 \theta\binom{n-l-1}{n-k-1}+\binom{n-l-1}{n-k}\right]\tan^l\theta \int_{\p\S_t} f\s_l^{\mathbb{S}^n}.
\end{eqnarray*}
Recall that $T$ satisfies \eqref{choiceT} and the flow \eqref{normal flow} induces a normal hypersurface flow $\p\S_t\subset \mathbb{S}^n$ with normal speed $\frac{f}{\sin\theta},$ that is, \begin{eqnarray*}
		\p_tx\big|_{\p\BB^n}=f\nu+f\cot\theta\mu=\frac{f}{\sin\theta}\bar{\nu}.
	\end{eqnarray*}
From Proposition \ref{fvf-sphere}, we know 
\begin{eqnarray*}
	\frac{d}{dt} W_{k}^{\mathbb{S}^n}(\widehat{\p\Sigma_t})=\frac{1}{\sin\theta}\binom{n}{k}^{-1}\int_{\p \S_t}f \sigma_k^{\mathbb{S}^n}.
\end{eqnarray*}
We conclude  that for $0\le k\le n-1$,
\begin{eqnarray*}
&&	\frac{d}{dt}\Bigg\{ \int_{\S_t}\sigma_k dA_t -\binom{n}{k} \cos\theta \sin^{k}\theta    W_k^{\SS^{n}}(\widehat{\p\S_t})  \\&&\quad -\cos^{k-1} \theta \sum_{l=0}^{k-1} (-1)^{k+l} \binom{n}{l}\left[\cos^{2} \theta \binom{n-l-1}{n-k-1}  +  \binom{n-l-1}{n-k}\right] \tan^l\theta   W_l^{\SS^n}(\widehat{\p\S_t})\Bigg\}
\\& =&(k+1)\int_{\S_t} \sigma_{k+1} f dA_t.
\end{eqnarray*}By using  
\begin{eqnarray*}
	\binom{n}{l}\cdot \binom{n-l-1}{n-k-1}=\binom{n}{k} \cdot \binom{k}{l} \frac{n-k}{n-l},
\end{eqnarray*}
and
\begin{eqnarray*}
	\binom{n}{l}\cdot \binom{n-l-1}{n-k}=\binom{n}{k} \cdot \binom{k}{l} \frac{k-l}{n-l},
\end{eqnarray*} we get \eqref{fvf-1}. 

For $k=n$, the computation is similar, with replace \eqref{sigma alge} by 
\begin{eqnarray*}
\s_n({h|\mu\mu})=\det(\sin \theta  \widehat{h}-\cos\theta I_{n-1})=(-\cos\theta)^{n-1} \det(I_{n-1}-\tan \theta \widehat{h}),
\end{eqnarray*}then it follows 
\begin{eqnarray*}
 \frac{d}{dt}\int_{\S_t} \sigma_n dA_t&=& 
-\int_{\p\S_t} \frac{f}{\sin\theta} \s_n^{\mu\mu}=\int_{\p\S_t} (-1)^{n} \frac{\cos^{n-1}\theta}{\sin\theta} f\det(I_{n-1}-\tan \theta \widehat{h})
 \\&=&\frac{\cos^{n-1} \theta}{\sin\theta}  \sum_{l=0}^{n-1} (-1)^{n+l}   \tan^l\theta  {\int_{\p\S_t} f\s_l^{\mathbb{S}^n} }.
\end{eqnarray*} We complete the proof.
\qed

\subsection{Estimates for convex hypersurfaces with $\theta$-capillary boundary}\

To begin with, we state a theorem, due to Ghomi \cite{Ghomi}, about global convexity of a compact, connected, immersed $C^2$-hypersurface with positive curvature.
\begin{thm}[Ghomi \cite{Ghomi}, Theorem 1.2.5]\label{Ghomi-thm} Let $\S\subset \mathbb{R}^{n+1}$ be a compact, connected, immersed $C^2$-hypersurface with positive curvature. Then $\S$ may be extended to a $C^2$-ovaloid if and only if for any boundary component $\Gamma$ and any $p\in \Gamma$, $\Gamma\cap T_p\S=\{p\}$.
\end{thm}
Recall that a closed immersed hypersurface $\S\subset \mathbb{R}^{n+1}$ is called an ovaloid if for any $p\in \S$,  $\S$ lies on one side of $T_p\S$. The classical Hadamard theorem says that a closed embedded hypersurface with positive curvature must be an ovaloid. Ghomi's theorem can be viewed as a generalization of Hadamard's theorem on hypersurface with boundary.
\begin{prop}\label{curv-radii-prop}
For any $p\in \overline{\mathbb{B}}^n$ and $z\in \widehat{\p\S}\subset\SS^n$, it holds $$\<z-x(p),\nu(p)\>\le 0,$$ with equality holding if and only if $z=x(p)$.
\end{prop}
\begin{proof} Consider $\p\S\subset \SS^n$. From Proposition \ref{basic-capillary}, we see that  $$(\widehat{ h}_{\alpha\beta})>\cot\theta (\delta_{\a\b})\ge 0.$$From a comparison theorem (cf. \cite[Theorem 0.1]{Ge}), we have \begin{eqnarray}\label{comparison}
\max_{y\in\p\S} {\rm dist}_{\SS^n}(y, \p\S)<\theta.
\end{eqnarray}

For each $p\in \overline{\mathbb{B}}^n$, we denote by $S_{p, \theta}$ the geodesic sphere in $\SS^n$ of radius $\theta$ passing through $x(p)$. It follows easily from \eqref{comparison} that \begin{eqnarray}\label{curv-radii-prop1}
\widehat{\p\S}\subset \widehat{S_{p, \theta}}\hbox{ with }\p\S\cap S_{p, \theta}=\{x(p)\}.
\end{eqnarray}
See Figure 1. 
	\begin{figure}[H] \includegraphics[width=0.35\linewidth]{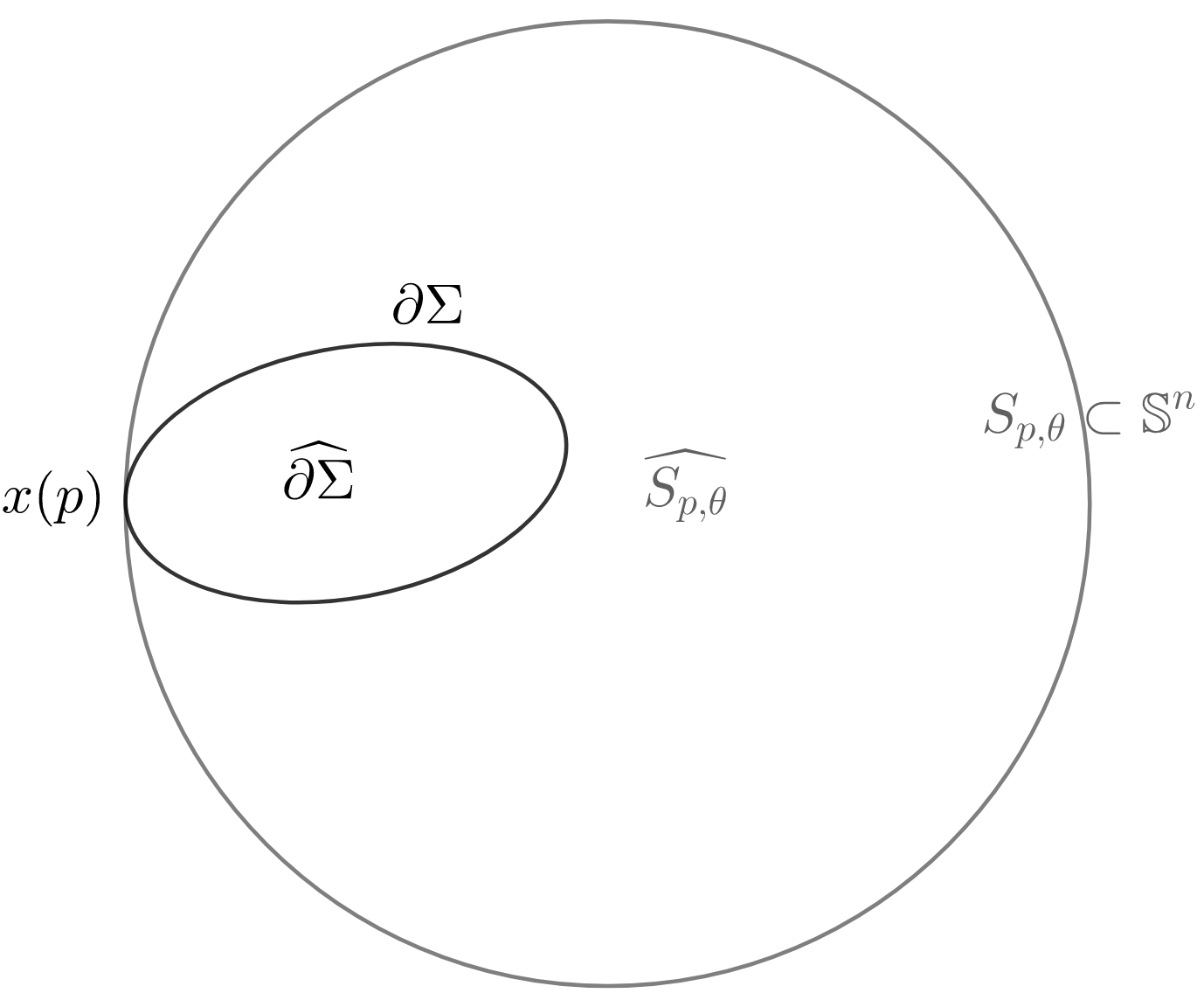} \label{fig0}  \caption*{Figure 1}   
\end{figure}
Since $\S$ is $\theta$-capillary, we see that $T_{p}\S\cap \SS^n=S_{p, \theta}$ for each $x(p)\in \overline{\mathbb{B}}^n$. Therefore, \eqref{curv-radii-prop1} implies that $\widehat{\p\S}\setminus\{x(p)\}$ lies strictly on one side of $T_{p}\S$, which leads to the assertion.
\end{proof}

Applying Ghomi's theorem and Proposition \ref{curv-radii-prop}, we have the following property for a strictly convex $\theta$-capillary hypersurface in the unit ball. 
\begin{prop}\label{hadamard thm for bdry}
 Let $\S\subset \overline{\mathbb{B}}^{n+1}$ be a strictly convex hypersurface with $\theta$-capillary boundary for $\theta\in (0,\frac{\pi}{2}]$, which is given by the embedding $x:\overline{\mathbb{B}}^n\to \overline{\mathbb{B}}^{n+1}$. Then for any $p\in \S$,  $\S$ lies on one side of $T_p\S$.
\end{prop}
\begin{proof}  Proposition \ref{curv-radii-prop} implies for $p\in \p\S$, $\p\S\cap T_p\S=\{p\}$.
By using Theorem \ref{Ghomi-thm},  we see that $\S$ may be extended to a $C^2$-ovaloid, in particular, our assertion follows.
\end{proof}

The following facts about strictly convex hypersurface with $\theta$-capillary boundary   will be used later. 
\begin{prop}\label{ellipticity} Let $\S\subset \overline{\mathbb{B}}^{n+1}$ be a strictly convex hypersurface with $\theta$-capillary boundary for $\theta\in (0,\frac{\pi}{2}]$, which is given by the embedding $x:\overline{\mathbb{B}}^n\to \overline{\mathbb{B}}^{n+1}$. 
Then there exists $e_0\in {\rm int}(\widehat{\p\S})$ and some $\delta_0>0$, depending on $\S$,  such that the following properties hold:
	\begin{enumerate}	
		\item\label{(1)} $		\langle x,e_0\rangle \ge \cos\theta+\delta_0$.
	
\item\label{(2)}  $\langle \nu,e_0\rangle<-\cos^2\theta$.
	
	\item\label{(3)}   $\langle \mu,e_0\rangle >0$ on $\p{\mathbb{B}}^n$.

\item\label{(4)} $
	\langle x+\cos\theta \nu,e_0\rangle \geq \delta_0>0 
$.

	\end{enumerate}
\end{prop}

\begin{proof}	The conclusion for $\theta=\frac{\pi}{2}$ has been proved by Lambert-Scheuer \cite[Section 4]{LS}. Hence we prove the case for $\theta\in(0,\frac{\pi}{2})$.

(1) {\bf Step 1.} Define $$\phi(p):=-\left<x(p),\nu(p)\right>, \quad p\in \overline{\BB}^n.$$
Note that $\phi|_{ \p \mathbb{B}^n}\equiv \cos\theta>0$. On the other hand, since $\mu$ is a principal direction and $\S$ is strictly convex, 
	$$D_\mu \phi =- h_{\mu\mu}\<x,\mu\>=-  h_{\mu\mu}\sin\theta<0,\quad   \hbox{ along } \p \mathbb{B}^n,$$
which implies that $\phi$ cannot attain its maximum on $\p \mathbb{B}^n$. Thus, $\phi$ attains its maximum at some interior point, say $p_0\in  \BB^n$. Therefore $$\phi(p_0)>\phi\big|_{ \p \mathbb{B}^n}= \cos \theta,$$
and at $p_0$, we have $$0=D_{e_i}\phi(p_0)=-h_{ij}(p_0)\left<x(p_0),e_j(p_0)\right>,$$ for any $e_i\in T_{p_0}\S$.
	Since $(h_{ij}(p_0))$ is positive definite, we get $\left<x(p_0),e_j(p_0)\right>=0$ for any $e_j\in T_{p_0}\S$. This implies $x(p_0)\parallel \nu(p_0)$. Since $\phi(p_0)>\cos \theta>0,$ for $\theta\in (0,\frac{\pi}{2})$, we obtain
 $$-\left<x(p_0),\nu(p_0)\right>=|x(p_0)|>\cos\theta.$$
From Proposition \ref{hadamard thm for bdry}, we see 
	$$\left<x(p)-x(p_0),\nu(p_0)\right>\le 0, \hbox{ for } p\in \overline{\BB}^n.$$
	Define $e_0:=-\nu(p_0)$, it follows $$\left<x(p), e_0\right>=\left<x(p), -\nu(p_0)\right>\ge -\left<x(p_0),\nu(p_0)\right> >\cos\theta,   $$for all $p\in \overline{\BB}^n$. See Figure 2. 
	\begin{figure}[H] \includegraphics[width=0.35\linewidth]{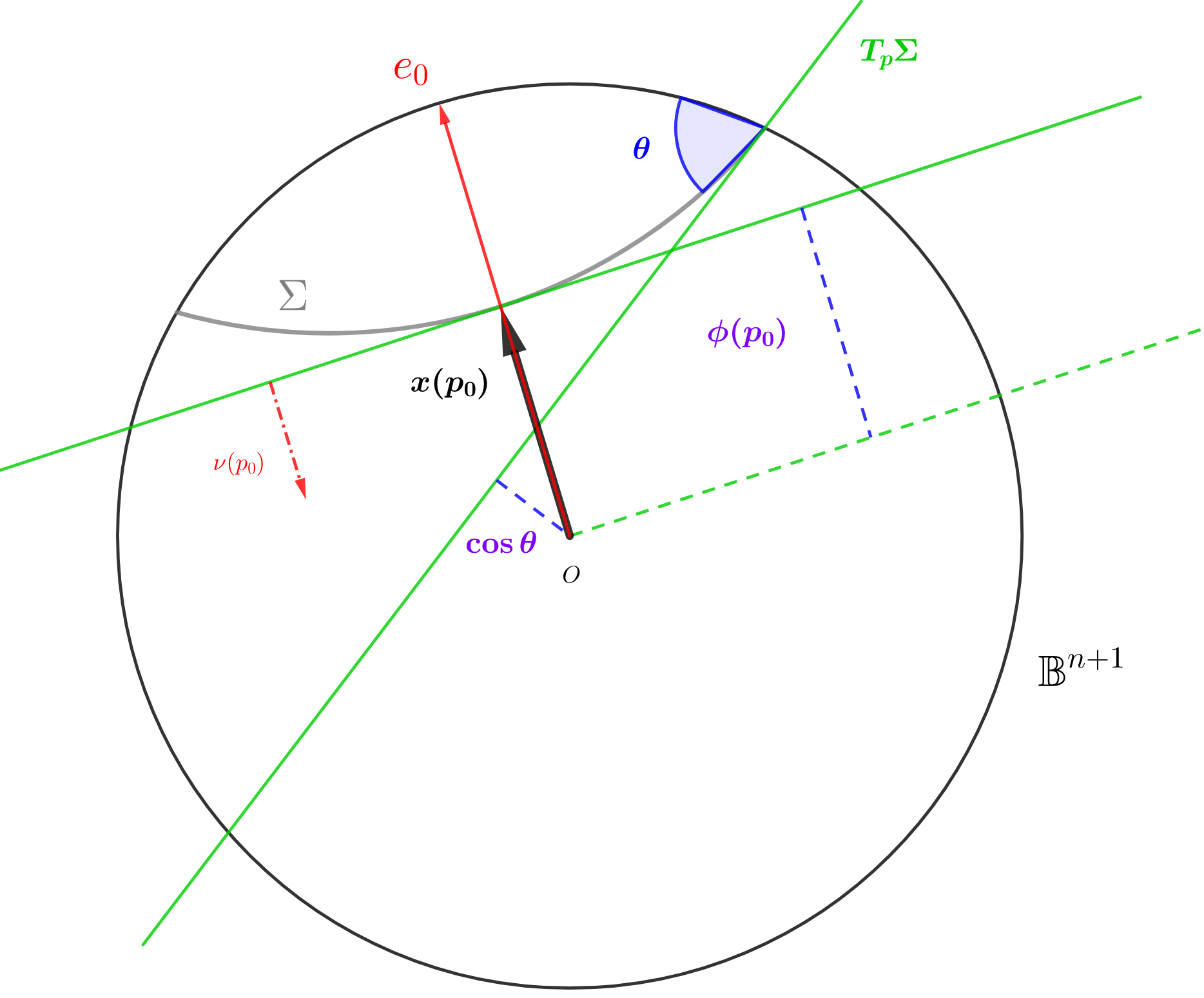} \label{fig1}  \caption*{Figure 2}   
\end{figure}
Next, we claim that $e_0$ must be contained in ${\rm int}(\widehat{\p \S})$. In fact, if not, there exists some $\tilde{p}\in \p  \BB^n$   such that $ x(\tilde{p})\in \p\S$ and  ${\rm dist}_{\SS^n}(e_0, x(\tilde{p}))={\rm  dist}_{\SS^n}(e_0, \p \S)\ge 0$. Using Proposition \ref{hadamard thm for bdry}, we see that $\S$ lies on one side of $T_{\tilde{p}}\S$, say $T_{\tilde{p}}\S^+$. By the fact  $\theta\in(0,\frac{\pi}{2})$, we see that $T_{\tilde{p}}\S^+$ does not contain the line segment $l$ through $e_0$ and the origin. However, $l$ goes through $x(p_0)\in \S$ since  $e_0=-\nu(p_0)=\frac{x(p_0)}{|x(p_0)|}$. This leads to a contradiction and we get the claim.

{\bf Step 2.} Denote the tangent hyperplane $ T_{p_0}\S:=\mathcal{H}$, where $p_0$ is determined in Step 1. It is known that $\mathcal{H}$ intersects $\mathbb{S}^{n}$ at the angle  $\theta_1<\theta$, where $\theta_1$ is given by
$$\cos\theta_1=|x(p_0)|>\cos\theta.$$

Now we parallel transport $\mathcal{H}$ downwards along the direction $-e_0$ until it achieves a position $\mathcal{H}'$ which satisfies that $\mathcal{H}'$ intersects with $\mathbb{S}^{n}$ at the   angle $\theta$, see Figure 3.
\begin{figure}[H] \includegraphics[width=0.35\linewidth]{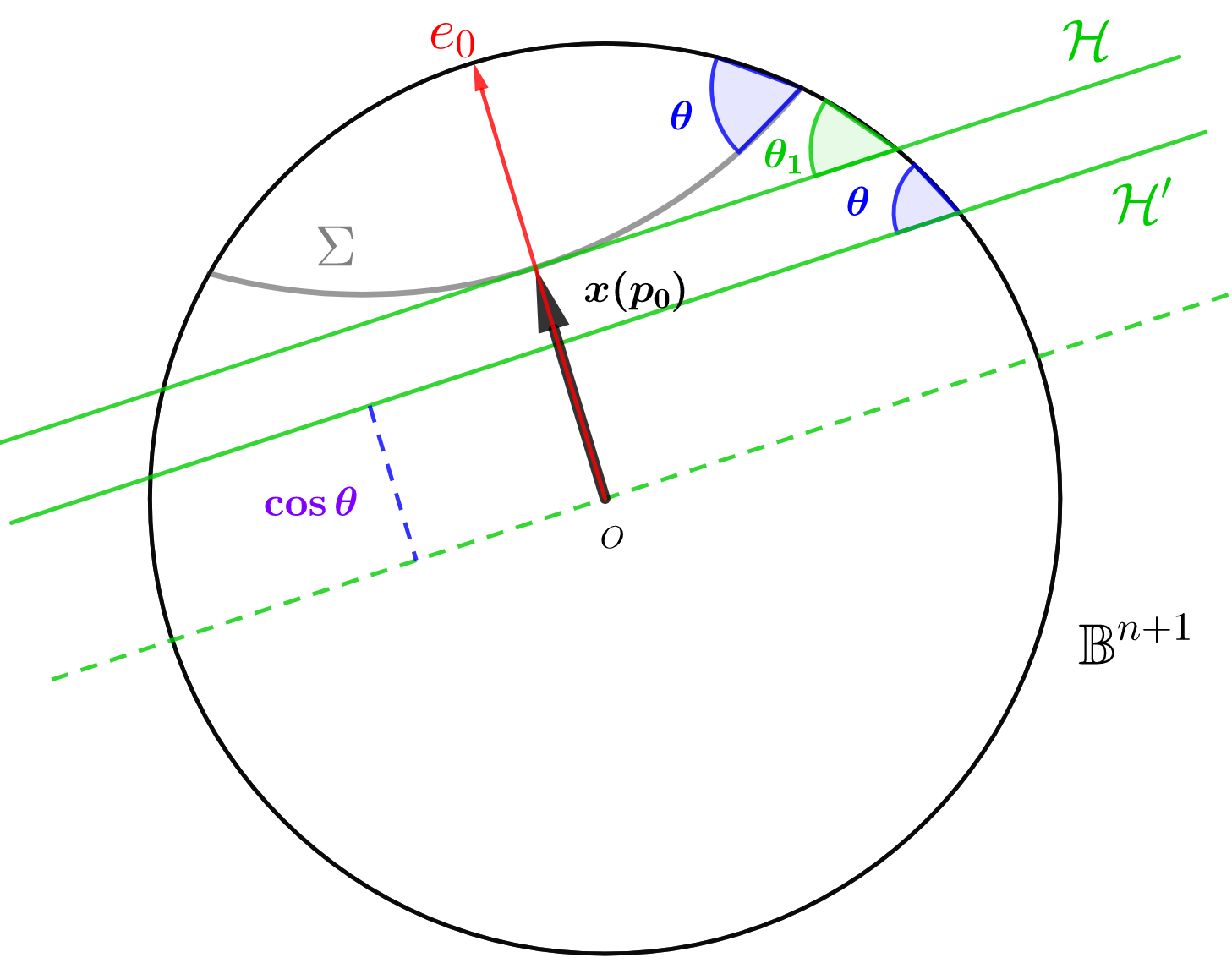} \label{fig2}  \caption*{Figure 3}   
	\end{figure}
We construct the foliation $C_{\theta, s}(e_0), s\ge 0$ so that $C_{\theta, \infty}(e_0):=\mathcal{H}'$. Recall $C_{\theta, s}(e_0), s>0$ is the spherical cap of radius $s$ around $e_0$  with $\theta$-capillary boundary, see \eqref{static model}.
	
	Since $C_{\theta, 0}(e_0)=\mathcal{H}'$ does not touch $\S$ and lies below $\S$, there exists some $s_0>0$, depending on $\theta_1$, such that $C_{\theta, s_0}(e_0)$ touches $\S$ at the first time. Hence $$\left<x, e_0\right>\ge \min_{x\in C_{\theta, 0}(e_0)}\left<x, e_0\right>:=\cos\theta+\delta_0,$$
for some $\delta_0>0$. From above, it is clear that $\delta_0$ depends on $s_0$ and in turn on $\S$. 

Hence we complete the proof of statement (\ref{(1)}).

\

(2) From (1), we know that $$\S\subset H_{e_0}:=\{\<x, e_0\>>0\}.$$ Since $e_0\in {\rm int}(\widehat{\p \S})$, we see that $$\widehat{\p \S}\subset H_{e_0}.$$ On the other hand, we see from Corollary \ref{convexity} that $\p\S$ is a strictly convex hypersurface in $\SS^n$. It follows from \cite{dCW,GX} that, there exist two disjoint open connected components $A$ and $B$, such that $$\SS^n\setminus\bar \nu(\p \BB^n)=A\cup B,$$ where $A$ is the interior of the strictly convex body in $\SS^n$ which $\bar \nu(\p \BB^n)$ bounds.
Since $\widehat{\p \S}\subset H_{e_0}$, we have $$\p A\subset H_{-e_0},$$
which implies $$\left<\bar \nu(p), e_0\right>\le 0, \quad p\in\p\BB^n.$$ Using \eqref{co-normal bundle}, we  get
 $$\left<\nu(p)+\cos\theta x(p) , e_0\right>\le 0, \quad p\in\p\BB^n.$$ Combining with $\left<x(p), e_0\right> > \cos\theta$ from (\ref{(1)}), it yields that \begin{eqnarray}\label{eq1}
	\left<\nu(p), e_0\right>\le -\cos\theta\left<x(p), e_0\right><-\cos^2\theta,\quad p\in\p\BB^n.
\end{eqnarray}
On the other hand, 
assume $$\mathbb{S}^n\setminus \nu(\p \BB^n)=\tilde{A}\cup \tilde{B},$$where $\tilde{A}$ is the component which contains $e_0$. In view of \eqref{eq1},  it holds that $$\left<y,  e_0\right> < -\cos^2\theta, \quad y\in \tilde{B}.$$ Since $\nu(\BB^n)$ is simply connected, $\nu(\BB^n)$ is either $\tilde{A}$ or $\tilde{B}$. However, $\nu(\BB^n)$ cannot contain $e_0$. Thus $\nu(\BB^n)=\tilde{B}$. Therefore  $$\left<\nu(p), e_0\right><  -\cos^2\theta, \quad p\in\overline{\BB}^n.$$

 \vspace{.4cm}
(3) Using equation \eqref{co-normal bundle} and the statement (\ref{(1)}),  we deduce
\begin{eqnarray*}
	\langle \mu(p),e_0\rangle =\frac{1}{\sin\theta}\langle  x(p)+\cos\theta \nu(p),e_0\rangle >\frac{1}{\sin\theta}\left(\cos\theta+\cos\theta \langle \nu(p),e_0\rangle \right)>0,
\end{eqnarray*}for all $p\in\p\BB^n$.
\vspace{.4cm}

(4) 	This follows directly from (1).
	\end{proof}	

\begin{prop}\label{star-shaped}
Let $\S\subset \overline{\mathbb{B}}^{n+1}$ be a strictly convex hypersurface with $\theta$-capillary boundary for $\theta\in (0,\frac{\pi}{2})$ and $e_0$ is given in Proposition \ref{ellipticity}. 
Then there exists $\ep>0$ depending on $\S$, such that
	\begin{eqnarray*}
		\langle X_{e_0},\nu\rangle  \ge \ep \cos^2 \theta.
	\end{eqnarray*}
\end{prop}
\begin{rem} In the case $\theta=\frac{\pi}{2}$, Lambert-Scheuer \cite{LS} showed that for  strictly convex hypersurface with free boundary and $e_0\in {\rm int}(\widehat{\p\S})$ such that $\widehat{\p\S}\subset H_{e_0}$, 
there exists some constant $\ep>0$ depending on $\S$ such that $\langle X_{e_0},\nu\rangle \ge \ep$. Our estimate in Proposition \ref{star-shaped} is weaker in the case $\theta=\frac{\pi}{2}$. Nevertheless, it is enough for use in the case $\theta\in (0,\frac{\pi}{2})$.
\end{rem}
\begin{proof}
 Since $\S$ is strictly convex and $e_0\in \widehat{\p\S}$, then \begin{eqnarray}\label{xeq1}
\langle x-e_0,\nu\rangle \ge 0.
\end{eqnarray}
Since $\S$ is compact and $e_0\not\in \S$,  there exists  $\ep>0$ such that $\frac12|x-e_0|^2\ge \ep$ on $\S$.
Recall that $X_{e_0}=\langle x, e_0\rangle x-\frac{|x|^2+1}{2}e_0$.  Thus 
	\begin{eqnarray*}
		\langle X_{e_0},\nu\rangle &=&\langle x,e_0\rangle \langle x,\nu\rangle-\frac{|x|^2+1}{2} \langle \nu,e_0\rangle
		\\&=&\langle x,e_0\rangle \langle x-e_0,\nu\rangle -\frac{|x-e_0|^2}{2} \langle \nu,e_0\rangle
		\ge \ep\cos^2\theta.
	\end{eqnarray*}where the last inequality follows from \eqref{xeq1} and Proposition \ref{ellipticity} (2).
\end{proof}

\

\section{Locally constrained curvature flow}\label{sect 3}
In this section, we study the flow. Let $\S_0:=\S\subset \overline{\BB}^{n+1}$ be a smooth, properly embedded,  strictly convex hypersurface with $\theta$-capillary boundary, given by the embedding $x_0: \overline{\mathbb{B}}^n\to \overline{\mathbb{B}}^{n+1}$. Choose $e_0$ as in Proposition \ref{ellipticity}, which depends only on $\S_0$. For simplicity of notation, we omit the subscription and write $e=e_0$.
  Let $\S_t\subset \overline{\BB}^{n+1}$ be a family of smooth, properly embedded,  strictly convex hypersurfaces with $\theta$-capillary boundary, given by $x(\cdot, t): \overline{\mathbb{B}}^n\to \overline{\mathbb{B}}^{n+1}$, evolving by 
 \begin{equation}\label{flow with capillary}
\left\{ \begin{array}{lll}
&\partial_t x(\cdot, t)  =f(\cdot, t) \nu(\cdot, t)+T(\cdot, t), \quad &
 \hbox{ in }\mathbb{\BB}^n\times[0,T),\\
&\langle \nu(\cdot, t),\overline{N}\circ x(\cdot,t)\rangle  = -\cos\theta 
 \quad & \hbox{ on }\partial \mathbb{\BB}^n\times [0,T),\\
& x(\cdot,0)  = x_0(\cdot) \quad & \text{ in }   \overline{\BB}^n.
 \end{array}\right.
 \end{equation}
 where  \begin{eqnarray}\label{speed}
 	f=\frac{\langle x+\cos\theta  \nu, e\rangle }{F}-\langle X_e,\nu\rangle,
 \end{eqnarray} with $$F=\frac{H_n}{H_{n-1}}=\frac{n\sigma_n }{\sigma_{n-1}},$$ and $T\in T\S_t$ such that \eqref{choiceT} holds.
The following property of the flow \eqref{flow with capillary} is crucial in this paper.
\begin{prop}\label{monotone along flow}
	As long as $\S_t$ is strictly convex and the flow \eqref{flow with capillary} exists, 
	$ W_{n,\theta}$ is preserved and $ W_{k,\theta}$ is non-decreasing  for $1\leq k\leq n-1$.
\end{prop}

\begin{proof}
	From  Proposition  \ref{minkowski formulae 1} and Theorem \ref{fvf}, we know that
	\begin{eqnarray*}
		\frac{d}{dt}  W_{n,\theta}(\widehat{\S_t})=\frac{1}{n+1}\int_{\S_t}  
		\Big[H_{n-1}\left(\langle x,e\rangle+\cos\theta\langle \nu,e\rangle \right)-H_n\langle X_e,\nu\rangle\Big]  = 0,
	\end{eqnarray*}
and for $1\leq k\leq n-1$,
		\begin{eqnarray*}
		\frac{d}{dt}   W_{k,\theta}(\widehat{\S_t})
&=&\frac{n+1-k}{n+1}\int_{\S_t} H_k	\Big[
		\frac{H_{n-1}}{H_{n}}\left(\langle x,e\rangle+\cos\theta\langle \nu,e\rangle \right)-\langle X_e,\nu\rangle \Big]
		\\&\geq &\frac{n+1-k}{n+1}\int_{\S_t}\Big[ H_{k-1} \left(\langle x,e\rangle+\cos\theta\langle \nu,e\rangle \right)- H_k \langle X_e,\nu\rangle\Big]=0,
	\end{eqnarray*}
where we have also used the Newton-MacLaurin inequality \eqref{NM-ineq} and Proposition Proposition \ref{ellipticity} (4).
\end{proof}
The aim of this section is the following existence and convergence result for the flow \eqref{flow with capillary}.
\begin{thm}\label{flow-result}
Let $\S_0\subset \overline{\BB}^{n+1}$ be a smooth, properly embedded,  strictly convex hypersurface with $\theta$-capillary boundary for some $\theta\in (0,\frac{\pi}{2})$, given by the embedding $x_0: \overline{\mathbb{B}}^n\to \overline{\mathbb{B}}^{n+1}$. Choose $e_0$ as in Proposition \ref{ellipticity}. Then  $x(\cdot, t)$ to the flow \eqref{flow with capillary} exists for $t\in [0,\infty)$.  Moreover, $x(\cdot, t)$ converges smoothly to the spherical cap around $a$ with $\theta$-capillary boundary, which has the same $W_{n,\theta}$ as $\S_0$.
\end{thm}

\subsection{Barriers}
First of all, we have the short time existence for the flow \eqref{flow with capillary}. This is because the strict convexity and the $\theta$-capillary boundary of $\S_0$ imply that $\S_0$ satisfies $\<X_e, \nu\>\ge \ep \cos^2\theta$ for some $\ep>0$ by Proposition \ref{star-shaped}. This allows us to transform the flow equation to be a scalar equation with oblique boundary condition, by using the M\"obius coordinate transformation see \cite{LS} or \cite{SWX}. Also from Proposition \ref{ellipticity} there exists some $\delta_0>0$ depending on $\S_0$ such that $$\langle x+\cos\theta\nu,e\rangle \ge \delta_0.$$ Thus the  scalar equation is a strictly parabolic equation. Then the short time existence follows from the standard parabolic theory. 

Let $T^*$ be the maximal time of smooth existence of a solution to \eqref{flow with capillary}. The positivity of $F$ implies that $\S_t$  is strictly convex up to $T^*$ for flow \eqref{flow with capillary}. 
From  Proposition \ref{ellipticity} (1),
It follows that  $$\S_0\subset \widehat{C_{\theta, \infty}(e)}=\{\<x, e\> > \cos\theta\}.$$
Thus there exists some $0<R_1<R_2<\infty$, such that 
$$\S_0\subset \widehat{C_{\theta, R_2}(e)}\setminus \widehat{C_{\theta, R_1}(e)}.$$
The family of $C_{\theta, r}(e)$ forms natural barriers of \eqref{flow with capillary}. 
\begin{prop}\label{barrier est} $\S_t, t\in [0, T^*)$ satisfies
	$$\S_t\subset \widehat{C_{\theta, R_2}(e)}\setminus \widehat{C_{\theta, R_1}(e)}.$$
	In particular, there exists some $\ep_0>0$, depending only only on $R_1$  and $R_2$, such that 
	\begin{eqnarray*}
	\cos\theta+ \ep_0\leq \langle x(p,t),e\rangle \leq 1-\ep_0, \quad \forall (p,t)\in\overline{\BB}^n\times [0,T^*).
	\end{eqnarray*}
\end{prop}
\begin{proof}
Recall that $C_{\theta, r}(e)$ satisfies \eqref{static-eq}. Thus for each $r>0$, it is a static solution to the flow \eqref{flow with capillary}. The assertion follows from the avoidance principle for strictly parabolic equation with capillary boundary condition (see \cite[Section 2.6]{AW} or \cite[Proposition 4.2]{WW}). The height estimate holds since $\langle x,e\rangle\ge \cos\theta+ \ep_0$ in $\widehat{C_{\theta, R_2}(e)}$, and  $\langle x,e\rangle\le 1- \ep_0$ in $\widehat{C_{\theta, R_1}(e)}$, for some $\ep_0>0$.
\end{proof}
We have the following direct consequence.
\begin{cor}\label{ellipticity of flow}
$\S_t, t\in [0, T^*)$ satisfies
\begin{eqnarray}
	 \left<x +\cos\theta \nu , e_0\right> \ge \ep_0.
\end{eqnarray}	\end{cor}

\subsection{Evolution equations}\

In order to avoid confusion with tensor indices, we abbreviate $X=X_e$, keeping in mind that $X$ depends on $e$. 
We introduce the operator
\begin{eqnarray*}
	&&\mathcal{L}:=\p_t -\frac{\langle x+\cos\theta  \nu, e\rangle }{F^2} F^{ij}\n^2_{ij}-\langle X-\frac{\cos\theta }{F}e +T,\n\rangle,
\end{eqnarray*}
Denote $\mathcal{F}:=\sum\limits_{i=1}^n F^i_i=\sum\limits_{i=1}^n \frac{\p F}{\p h^i_i}$.

For $F=\frac{n\sigma_n}{\sigma_{n-1}}$, use Proposition \ref{prop2.2}, we have
\begin{eqnarray}
&&	\mathcal{F}-\frac{F^{ij}h_{ij}}{F}= \mathcal{F}-1\geq 0,\label{xeq11}\\
&&\frac{F^{ij}h_i^k h_{kj}}{F^2}= 1.\label{xeq22}
\end{eqnarray}

\begin{prop}\label{evol of x-a}
	Along the flow \eqref{flow with capillary}, we have 
\begin{eqnarray*}
	\mathcal{L}(\langle x,e\rangle)& =&2F^{-1}\langle x+\cos\theta \nu,e\rangle \langle \nu,e\rangle +\cos\theta F^{-1}|e^T|^2-\langle  X,e\rangle,\end{eqnarray*}
	where $e^T$ is the tangential projection of $e$ on $\Sigma_t$.\end{prop} 
\begin{proof}
\begin{eqnarray*}
\p_t\langle  x,e\rangle= \left(\frac{\langle x+\cos\theta \nu,e\rangle }{F}-\langle X,\nu\rangle \right)\langle \nu,e\rangle +\<T, e\>.
\end{eqnarray*}
\begin{eqnarray*}
F^{ij}\n^2_{ij} \langle x,e\rangle= -F^{ij}h_{ij}\langle  \nu,e\rangle= -F\langle  \nu,e\rangle.
\end{eqnarray*}
Thus 
\begin{eqnarray*}
\mathcal{L}(\langle x,e\rangle)& =&\langle \p_t x,e\rangle  -\frac{\langle x +\cos\theta   \nu,e\rangle }{F^2} F^{ij}\n^2_{ij} \langle x,e\rangle -\left< X-\frac{\cos\theta }{F}e+T,\n \langle x,e\rangle \right>
	\\&=&2F^{-1}\langle x+\cos\theta \nu,e\rangle \langle \nu,e\rangle +\cos\theta F^{-1}|e^T|^2-\langle  X,e\rangle.
\end{eqnarray*}
\end{proof}

\begin{prop}\label{evol of F}
	Along the flow \eqref{flow with capillary}, we have 
\begin{eqnarray}\label{evol-F}
	\mathcal{L} F& =&\frac{2F^{ij}}{F^2} \langle x +\cos\theta   \nu,e\rangle_{;i} F_{;j}\nonumber
	\\&&- \frac{2  \langle x  +\cos\theta  \nu,e\rangle  }{F^3} F^{ij} F_{;i}F_{;j}	+(1-\mathcal{F})\langle \nu,e\rangle.
\end{eqnarray}
and \begin{eqnarray}\label{neumann bdry of F}
\nabla_\mu F=0\hbox{ along }\p\S_t.
\end{eqnarray}
\end{prop} 

 \begin{proof}
From Proposition \ref{basic evolution eq}, we have
\begin{eqnarray}\label{p t F} 
			\p_t F =-F^{ij}f_{;ij}-fF^{ij}h_{i}^kh_{kj} +\n_T F.
\end{eqnarray} 
Next we compute the first term in the right hand side above.  
\begin{eqnarray*}
 -F^{ij}	f_{;ij}&=&-F^{ij}\left(\frac{\langle x+\cos\theta \nu,e\rangle }{F}-\langle X,\nu\rangle\right)_{;ij}
	\\&=&\frac{\langle x,e\rangle }{F^2}F^{ij}F_{;ij} +  \frac{ 2F^{ij}\langle x,e\rangle_{;i} F_{;j}}{F^2}-\frac{2\langle x,e\rangle }{F^3} F^{ij}F_{;i}F_{;j}  +\frac{F^{ij}h_{ij}}{F} \langle \nu,e\rangle\\&&+ 
	\frac{\cos\theta \langle \nu,e\rangle }{F^2} F^{ij}F_{;ij}+\cos\theta  \frac{2F^{ij}\langle \nu,e\rangle_{;i} F_{;j}}{F^2}  \\&&-\cos\theta \frac{2\langle \nu,e\rangle }{F^3} F^{ij}F_{;i}F_{;j} +\left( \cos\theta  \frac{F^{ij}h^k_ih_{kj} \langle \nu,e\rangle }{F}-\cos\theta \frac{F_{;k}\langle e_k,e\rangle }{F} \right)\\&& +F^{ij}\langle X,\nu\rangle_{;ij}.
\end{eqnarray*}
Similarly, by conducting a direct computation, (one can also refer to \cite[Lemma 3.3]{SWX}), we know that
\begin{eqnarray}
		\langle X,\nu\rangle_{;ij} &=&h_{i;j}^k \langle X,e_k\rangle +\langle x,e\rangle h_{ij}-h^k_ih_{kj} \langle X,\nu\rangle -\langle e,\nu\rangle g_{ij} \nonumber\\&&+h^k_j\left( \langle e_i,e\rangle \langle x,e_k\rangle -\langle e_i,x\rangle \langle e,e_k\rangle \right)
		 \nonumber\\&&+h_i^k \left(\langle e_j,e\rangle \langle x,e_k\rangle-\langle e_j,x\rangle \langle e,e_k\rangle \right),\label{x nu}
\end{eqnarray}
which yields that
\begin{eqnarray*}
F^{ij}\langle X,\nu\rangle_{;ij}
	&=&F^{ij}\left(h_{i;j}^k \langle X,e_k\rangle +h_{ij}\langle x,e\rangle-g_{ij}\langle \nu,e\rangle -h^k_ih_{kj} \langle X,\nu\rangle \right)
	\\&=&F_{;k}\langle X,e_k\rangle +\left(F\langle x,e\rangle -\mathcal{F}\langle \nu,e\rangle\right) -F^{ij}h_i^kh_{kj}\langle X,\nu\rangle .
\end{eqnarray*}
And we have
\begin{eqnarray*}
&& -f F^{ij}h^k_ih_{kj}= -\left(\frac{\langle x+\cos\theta \nu,e\rangle }{F}-\langle X,\nu\rangle\right) F^{ij}h^k_ih_{kj}.
\end{eqnarray*}
Pluging those terms into equation \eqref{p t F}, after simplifications,   we conclude that
\begin{eqnarray*}
	\p_t F&=&\Big[ \frac{\langle x+\cos\theta \nu ,e\rangle }{F^2}F^{ij}F_{;ij}+F_{;k}\langle X,e_k\rangle \Big]+\frac{2F^{ij}F_{;j}}{F^2}\left( \langle x,e\rangle_{;i}+\cos\theta   \<\nu,e\>_{;i}\right)
	\\&& - \frac{2  \langle x +\cos\theta  \nu,e\rangle }{F^3} F^{ij} F_{;i}F_{;j}-\cos\theta \frac{F_{;k}\langle e_k,e\rangle }{F}+\left(\frac{F^{ij}h_{ij}}{F}-\mathcal{F}\right)\langle \nu,e\rangle\\&& +F\langle x,e\rangle \left(1-\frac{F^{ij}h_i^k h_{kj}}{F^2}\right)+\n_T F,
\end{eqnarray*}
therefore, we obtain 
\begin{eqnarray*}
	\mathcal{L} F& =&\frac{2F^{ij}}{F^2}\langle x +\cos\theta   \nu,e\rangle_{;i} F_{;j}
	- \frac{2 \langle x +\cos\theta   \nu,e\rangle  }{F^3} F^{ij} F_{;i}F_{;j}	\\&&+\left(\frac{F^{ij}h_{ij}}{F}-\mathcal{F}\right)\langle \nu,e\rangle  +F\langle x,e\rangle \left(1-\frac{F^{ij}h_i^k h_{kj}}{F^2}\right).
\end{eqnarray*}
 In view of \eqref{xeq11} and \eqref{xeq22}, we complete the proof of the evolution equation.
 
 Next we prove $\nabla_\mu F=0$.  From Proposition \ref{robin bdry of speed1}, we have along $\p\S_t$, 
 $$\n_\mu f= \left(\frac{1}{\sin\theta} +\cot\theta h_{\mu\mu}\right) f.$$
 From \cite[Proposition 3.3]{WX1}, we know along $\p\S_t$, 
\begin{eqnarray*}
	\n_\mu \langle X,\nu\rangle =\left(\frac{1}{\sin\theta} +\cot\theta h_{\mu\mu}\right)\langle X,\nu\rangle,
\end{eqnarray*}
and
\begin{eqnarray*}
	\n_\mu \langle x+\cos\theta \nu,e\rangle =\left(\frac{1}{\sin\theta} +\cot\theta h_{\mu\mu}\right) \langle x+\cos\theta \nu,e\rangle.
\end{eqnarray*}
From the definition \eqref{speed} of $f$, we conclude that $\nabla_\mu F=0$.
\end{proof}
\begin{prop}\label{evol of H} 	Along the flow \eqref{flow with capillary},  we have 
	\begin{eqnarray}\label{evol-H} 
		\mathcal{L}H&=&	H \langle x,e\rangle +\left(HF^{-1}-n+\cos\theta |h|^2 F^{-1}\right)\langle \nu,e\rangle\nonumber
		\\&&+ \langle x+\cos\theta \nu,e\rangle \left( H -2|h|^2 F^{-1}-2F^{-3} |\n F|^2\right)\nonumber\\&& +  2F^{-2}\langle \n F,e\rangle+2\cos\theta F^{-2}g^{ij}F_{;j} h_{i}^k\langle e_k,e\rangle \nonumber
		\\&& +F^{-2}   \langle x+\cos\theta \nu ,e\rangle g^{ij}h_{kl;i}h_{st;j}\frac{\p^2 F}{\p h_{kl}\p h_{st}},
\end{eqnarray}
and \begin{eqnarray}\label{boundaryH}
\nabla_\mu H\le 0\hbox{ along }\p\S_t.
\end{eqnarray}

\end{prop}

\begin{proof}
By Proposition \ref{basic evolution eq}, we have
	\begin{eqnarray}\label{p_t H}
			\p_t H&=&-\Delta  f-|h|^2 f+ \n_T H\nonumber
		=-\Delta  \Big[\frac{\langle x +\cos\theta   \nu,e\rangle }{F}- \langle X_e,\nu\rangle\Big]\nonumber
		\\&&-|h|^2\left(\frac{\langle x +\cos\theta  \nu,e\rangle }{F}- \langle X_e,\nu\rangle\right)+\n_T H.
	\end{eqnarray}Note that, from equation \eqref{x nu}, it holds
	\begin{eqnarray*}
\Delta   \langle X,\nu\rangle  = H_{;k}\langle X,e_k\rangle +\langle x,e\rangle H-|h|^2 \langle X,\nu\rangle -n \langle \nu,e\rangle  ,
	\end{eqnarray*}
and
	\begin{eqnarray*}
-\Delta \left(\frac{\langle x+\cos\theta \nu, e\rangle }{F}\right)
		&=&\frac{H}{F}\langle \nu,e\rangle -\frac{\cos\theta}{F}\left(\langle \n H, e\rangle-|h|^2 \langle \nu,e\rangle \right) \\&&+2F^{-2}\left(\langle \n F,e\rangle+\cos\theta g^{ij}F_{;j} h_{i}^k\langle e_k,e\rangle \right)\\&&+\Delta F \cdot  F^{-2}  \langle x+\cos\theta \nu ,e\rangle -2F^{-3}   \langle x+\cos\theta \nu,e\rangle |\n F|^2.
	\end{eqnarray*}Next we  compute the term containing $\Delta F $. 
Due to the Codazzi equations, the Ricci identities and the Gauss equation, there holds (see for example \cite{SWX})
\begin{eqnarray*}
h_{kl;ij}=h_{ij;kl}+(h_{la}h_{jk}-h_{lk}h_{ja})h_{i}^a+(h_{la}h_{ji}-h_{li}h_{ja})h_{k}^a,
\end{eqnarray*}
and thus
\begin{eqnarray*}
g^{ij}F^{kl}h_{kl;ij}=F^{kl}H_{;kl}+F^{kl}h_{la}h^{a}_{k}H-F|h|^{2}.
\end{eqnarray*}
It follows that
	\begin{eqnarray*}
		\Delta F&=&g^{ij}\left(\frac{\p^2 F}{\p h_{kl}\p h_{st}} h_{kl;i}h_{st;j}+\frac{\p F}{\p h_{kl}}h_{kl;ij}\right)
	\\&=&F^{kl}H_{;kl}+HF^{kl}h_{la}h^{a}_{k}-F|h|^2+g^{ij}h_{kl;i}h_{st;j}\frac{\p^2 F}{\p h_{kl}\p h_{st}}.
	\end{eqnarray*}
By adding above terms together into equation \eqref{p_t H}, we get the evolution equation for $H$.

Next we prove $\nabla_\mu H\le 0$ on $\p\S_t$. Let $\{e_{\a}\}_{\a=1}^{n-1}$ be an orthonormal frame of $T(\p \S_t)$, such that the second fundamental form of $\S_t$ is diagonal with respect to $\{e_{\a},\mu \}_{\a=1}^{n-1}$.
From  \eqref{neumann bdry of F}, we have $$0=\n_\mu F=F^{\mu\mu}h_{\mu\mu;\mu}+\sum\limits_{\alpha=1}^{n-1}F^{\alpha\alpha}h_{\alpha\alpha;\mu}.$$ It follows that
\begin{eqnarray*}
	\n_\mu H&=&h_{\mu\mu;\mu}+\sum_{\alpha=1}^{n-1}h_{\alpha\alpha;\mu}= - \sum_{\alpha=1}^{n-1} \frac{F^{\alpha\alpha}}{F^{\mu\mu}}h_{\alpha\alpha;\mu}+\sum_{\alpha=1}^{n-1} h_{\alpha\alpha;\mu}
	\\&=& \sum_{\alpha=1}^{n-1}  \frac{1}{F^{\mu\mu}}\big( F^{\mu\mu}-F^{\alpha\alpha}\big)(h_{\mu\mu}-h_{\alpha\alpha})\tilde{h}_{\a\a}.
\end{eqnarray*}
In the last equality we have used Proposition \ref{basic-capillary} (4). 
Since $F$ is a concave function, 
$$\big(F^{\mu\mu}-F^{\alpha\alpha}\big)(h_{\mu\mu}-h_{\alpha\alpha})\le 0, \quad \forall \a.$$
Also, $\tilde{h}_{\a\a}\ge 0$ due to Corollary \ref{convexity}.
It follows that 
\begin{eqnarray*}
	&&\n_\mu H= \sum_{\alpha=1}^{n-1}  \frac{1}{F^{\mu\mu}}\big( F^{\mu\mu}-F^{\alpha\alpha}\big)(h_{\mu\mu}-h_{\alpha\alpha})\tilde{h}_{\a\a}\le 0.
\end{eqnarray*}

\end{proof}

\subsection{Curvature estimates}\	

 First we have the uniform lower bound of $F$.
\begin{prop}
	Along the flow \eqref{flow with capillary}, it holds
	\begin{eqnarray*}
		F(p ,t)\geq \min_{\overline{\mathbb{B}}^n} F(\cdot, 0), \quad \forall (p,t)\in \overline{\mathbb{B}}^n\times [0,T^*).
	\end{eqnarray*}
\end{prop}
\begin{proof}

From Proposition \ref{evol of F},  \eqref{xeq11} and Proposition \ref{ellipticity} (2), we know
\begin{eqnarray*}
	\mathcal{L} F \geq 0, \quad \text{ mod} \quad \n F,
\end{eqnarray*}taking into account \eqref{neumann bdry of F}, the assertion follows from the maximum principle.
	\end{proof}
In particular, since $F=\frac{n\s_n}{\s_{n-1}}$, from the uniform lower bound of $F$, we get uniform curvature positive lower bound.
\begin{cor}\label{convexity preserve}
$\Sigma_t, t\in [0,T^*)$ is uniformly convex, that is, there exists $c>0$ depending only on $\S_0$, such that the principal curvatures of $\S_t$, $$\min_i \kappa_i(p, t) \ge c, \quad \forall (p,t)\in \overline{\mathbb{B}}^n\times [0,T^*).$$
\end{cor}
Then by Proposition \ref{star-shaped}, we get the following 
\begin{cor}\label{star-shaped preserve}
Along the flow \eqref{flow with capillary}, there exists $c>0$ depending only on $\S_0$, such that
$$\<X_e, \nu\> \ge c, \quad \forall (p,t)\in \overline{\mathbb{B}}^n\times [0,T^*).$$
\end{cor}
Next we obtain the uniform upper bound of $F$.
\begin{prop}\label{upper bounds F}Along the flow \eqref{flow with capillary}, there exists $C>0$ depending only on $\S_0$, such that
	\begin{eqnarray*} 
		F(p,t)\leq C,\quad \forall (p,t)\in \overline{\mathbb{B}}^n\times [0,T^*).
	\end{eqnarray*} 
\end{prop}	
\begin{proof}
Consider the function
	\begin{eqnarray*}
		\Phi:=\log F-\alpha   \langle x,e\rangle  
	\end{eqnarray*}where $\alpha>0$ will be determined later. Using \eqref{neumann bdry of F} and Proposition \ref{ellipticity} (3), we have on $\p\S_t$, 
\begin{eqnarray*}
	&&\n_\mu \Phi =-\alpha \n_\mu \langle x,e\rangle =-\alpha  \langle \mu ,e\rangle <0.
\end{eqnarray*} Thus $\Phi$ attains its maximum value at an interior point, say $p_0\in \BB^n$, then we have
\begin{eqnarray*}
	\n 		\Phi(p_0)=0,\quad \n^2\Phi(p_0) \leq 0, \quad \p_t \Phi(p_0)\ge 0.
\end{eqnarray*}
 Now all the computation below are conducted at the point $p_0$. We have 
\begin{eqnarray}\label{critical pt}
(	\log F)_{;i}=\alpha \langle x,e\rangle_{;i}, \quad\text{ and }\quad  \mathcal{L}		\Phi \geq 0.
\end{eqnarray}
From \eqref{evol-F} and \eqref{critical pt},
\begin{eqnarray*}
	\mathcal{L} \log F
	&=&F^{-1}\Big[\frac{2F^{ij}}{F^2} \langle x +\cos\theta \nu,e\rangle_{;i} F_{;j}
	- \frac{2 \langle x +\cos\theta \nu,e\rangle}{F^3} F^{ij} F_{;i}F_{;j}	+(1-\mathcal{F})\langle \nu,e\rangle  \Big]\\&&+\frac{\langle x+\cos\theta \nu,e\rangle }{F^2}F^{ij} (\log F)_{;i}(\log F)_{;j}
	\\&=&F^{-2}\left( {-\alpha^2 \langle x+\cos\theta \nu,e\rangle +2\alpha} \right)\cdot F^{ij} \langle x ,e\rangle_{;i}  \langle x ,e\rangle_{;j} \\&& + F^{-1}(1-\mathcal{F})\langle \nu,e\rangle  +2\cos\theta \alpha F^{-2}F^{ij}\langle x,e\rangle_{;j}\langle \nu,e\rangle_{;i}.
\end{eqnarray*}
By choosing orthonormal frame with principal directions, we may assume that  $h_{ij}$ is diagonal at $p_0\in \BB^n$. By $(h_{ij})>0$,we have \begin{eqnarray*}
2\cos\theta \alpha F^{-2} F^{ij}\langle x,e\rangle_{;j}\langle \nu,e\rangle_{;i}&=&2\alpha\cos\theta  F^{-2} F^{ij} \langle e_j,e\rangle h_{ik}\langle e_k,e\rangle \\&\leq & 2\alpha \cos\theta  F^{-2}\left|F^{ii}h_{ii}\right|=2\alpha \cos\theta F^{-1}.
\end{eqnarray*} 
Also by $(h_{ij})>0$ we have $F^{ii}\le 1$ for each $i$. Thus
$$F^{-1}(1-\mathcal{F})\langle \nu,e\rangle \le CF^{-1}.$$
By choosing $\alpha$ large, we have
\begin{eqnarray*}
	\mathcal{L} \log F\le CF^{-1}.
	\end{eqnarray*}
Combining with Proposition \ref{evol of x-a}, we get
\begin{eqnarray}\label{xeq33}
0&\leq& 	\mathcal{L}\Phi(x_0)  =\mathcal{L}\left( \log F-\alpha   \langle x,e\rangle  \right)\nonumber
	\\&\leq&CF^{-1}-\alpha\Big[2F^{-1}\langle x+\cos\theta \nu,e\rangle \langle \nu,e\rangle +\cos\theta F^{-1}|e^T|^2-\langle  X,e\rangle\Big] \nonumber
	\\&\le & CF^{-1}+ C\alpha F^{-1}+ \alpha \langle X,e\rangle.
\end{eqnarray}
By using Proposition \ref{barrier est}, we have  \begin{eqnarray*}
	&&\langle X,e\rangle =\langle x,e\rangle^2-\frac{1}{2}(|x|^2+1)
	=-\frac{1}{2}(1-\langle x,e\rangle^2)-\frac{1}{2}(1-\langle x,e\rangle^2) \leq -c_0,
\end{eqnarray*}  for some $c_0>0$.
The upper bound for $F$ follows from \eqref{xeq33}.
\end{proof}
Next we obtain the uniform bound of the mean curvature.
\begin{prop}\label{mean-curv-bound}
	Along the flow \eqref{flow with capillary}, there exists $C>0$ depending only on $\S_0$, such that
	\begin{eqnarray*} 
		H(p,t)\leq C,\quad \forall (p,t)\in \overline{\mathbb{B}}^n\times [0,T^*).
	\end{eqnarray*} 
\end{prop}

\begin{proof}
Firstly, from \eqref{boundaryH}, we know that $\n_\mu H\leq 0$ on $\p\S_t$. Thus $H$ attains its  maximum value at some interior point, say $p_0\in \BB^n$. We conduct all the computation below at $p_0$. 

Recall the evolution equation \eqref{evol-H} of $H$. From the concavity of $F=\frac{n\s_n}{\s_{n-1}}$ and Corollary \ref{ellipticity of flow}, we see
$$F^{-2}   \langle x+\cos\theta \nu ,e\rangle g^{ij}h_{kl;i}h_{st;j}\frac{\p^2 F}{\p h_{kl}\p h_{st}}\le 0. $$
Thus 
\begin{eqnarray*}
&&	0\leq	\mathcal{L}H(p_0) \le \text{J}_1+\text{J}_2+ \text{J}_3,
\end{eqnarray*}
where \begin{eqnarray*}
 \text{J}_1&:=&H\left(\langle x,e\rangle +\langle x+\cos\theta \nu,e\rangle \right)+(HF^{-1}-n)\langle \nu,e\rangle,\\
 \text{J}_2&:=&2F^{-2}\langle \n F,e\rangle -2\varepsilon F^{-3}\langle x+\cos\theta\nu,e\rangle |\n F|^2,\\
 \text{J}_3&:=&\cos\theta |h|^2 F^{-1}\langle \nu,e\rangle -2|h|^2 F^{-1}\langle x+\cos\theta\nu,e\rangle\\&&-2(1-\varepsilon)F^{-3}\langle x+\cos\theta\nu,e\rangle |\n F|^2 +   {2\cos\theta F^{-2}g^{ij} F_{;i}h_{jk}\langle e_k,e\rangle}.
\end{eqnarray*}
Here $\varepsilon>0$ will be chosen later.

By Cauchy-Schwarz inequality and the bound for $F$, we see easily that \begin{eqnarray*}
 \text{J}_1\le CH, \quad  \text{J}_2\le C_{\varepsilon}.
\end{eqnarray*}

Next we tackle $\text{J}_3$. By choosing orthonormal frame with principal directions, we may assume that  $g_{ij}=\delta_{ij}$ and $h_{ij}$ is diagonal at $p_0\in \BB^n$. Then
\begin{eqnarray*}
	F\cdot  \text{J}_3&=&
	-\Big[\langle x,e\rangle +\langle x+\cos\theta \nu,e\rangle\Big]|h|^2\\&&-2(1-\varepsilon)F^{-2}\langle x+\cos\theta \nu,e\rangle |\n F|^2+2\cos\theta F^{-1}  \sum_{i=1}^n h_{ii}F_{;i} \langle e_i,e\rangle.
\\&=&-W\sum_{i=1}^n \left(F_{;i}-\frac{V_i}{2W}h_{ii}\right)^2+\frac{1}{W}\sum_{i=1}^n\left(\frac{V^2_i}{4}-UW\right)h_{ii}^2,
\end{eqnarray*}
where for notation simplicity we used
\begin{eqnarray*}
	U:=\langle x,e\rangle +\langle x+\cos\theta \nu,e\rangle,\quad W:=2(1-\varepsilon)F^{-2}\langle x+\cos\theta \nu,e\rangle , 
\end{eqnarray*}and \begin{eqnarray*}
V_i:=2F^{-1}\cos\theta \langle e_i,e\rangle. 
\end{eqnarray*}

Next, we analyze the term $V^2_i-4UW$. For each $i$,
\begin{eqnarray*}
	V^2_i-4UW&=&4F^{-2}\Big[\cos^2\theta \langle e_i,e\rangle^2-2(1-\varepsilon)\langle x+\cos\theta \nu,e\rangle\left(\langle x,e\rangle +\langle x+\cos\theta \nu,e\rangle\right)   \Big]
	\\&\le &4F^{-2}\Big[\cos^2\theta |e^T|^2-2(1-\varepsilon)\langle x+\cos\theta \nu,e\rangle\left(\langle x,e\rangle +\langle x+\cos\theta \nu,e\rangle\right)   \Big].
\end{eqnarray*}
Recall that we have $$\langle x,e\rangle \geq \cos\theta+\ep_0= \frac{1}{\sqrt{1-\ep_1}} \cos\theta$$ for some $\ep_1>0$, and $$\langle x+\cos\theta\nu,e\rangle \geq \ep_0>0,$$ due to Proposition \ref{barrier est},  and Corollary \ref{ellipticity of flow}.  Then  
\begin{eqnarray*}
	&&2(1-\varepsilon)\langle x+\cos\theta \nu,e\rangle\left(\langle x,e\rangle +\langle x+\cos\theta \nu,e\rangle\right)  + \cos^2\theta \langle \nu,e\rangle^2
	\\&=&3(1-\varepsilon)\langle x+\cos\theta\nu,e\rangle^2+(1-\varepsilon)\langle x,e\rangle^2 +\varepsilon\cos^2\theta \langle \nu,e\rangle^2
	\\&\geq& 3(1-\varepsilon)\ep_0^2+\frac{1-\varepsilon}{1-\ep_1}\cos^2\theta.
\end{eqnarray*}
By choosing $\varepsilon=\ep_1$ in the  above inequality, we obtain
\begin{eqnarray*}
&&\cos^2\theta |e^T|^2-2(1-\varepsilon)\langle x+\cos\theta \nu,e\rangle\left(\langle x,e\rangle +\langle x+\cos\theta \nu,e\rangle\right)\le -3(1-\ep_1)\ep_0^2.
\end{eqnarray*}
It follows that 
	\begin{eqnarray*}
	V^2_i-4UW\leq -12(1-\ep_1)\ep^2_0F^{-2},\end{eqnarray*}
and hence
\begin{eqnarray*}
	\text{J}_3\le -C|h|^2.\end{eqnarray*}
Therefore,
\begin{eqnarray*}
	&&0 \leq \mathcal{L}H \leq \text{J}_1+\text{J}_2+ 	\text{J}_3\le CH+C-C|h|^2.\end{eqnarray*}
We conclude that $H$ is bounded above.
\end{proof}
It follows from Proposition \ref{mean-curv-bound} and Corollary \ref{convexity preserve} that all the principal curvatures are bounded.
\begin{cor}\label{curvature bound}
$\Sigma_t, t\in [0,T^*)$, has uniform curvature bound, that is, there exists $C>0$ depending only on $\S_0$, such that the principal curvatures of $\S_t$, $$\max_i \kappa_i(p, t) \le C, \quad \forall (p,t)\in \overline{\mathbb{B}}^n\times [0,T^*).$$
\end{cor}

\subsection{Convergence of the flow}\

\begin{prop}\label{global existence}
The flow  \eqref{flow with capillary} exists for all time with uniform $C^{\infty}$-estimates.
\end{prop}
 
 \begin{proof} Because of Corollary \ref{star-shaped preserve},  we can reduce the flow \eqref{flow with capillary} to a scalar parabolic equation with oblique boundary condition, by 
 introducing a conformal transformation map as in \cite{SWX, WW, WX2}. 
 
 Without loss of generality, we may assume $e=E_{n+1}=(0,\cdots, 0, 1)$. Define the transformation map
 	\begin{equation*}
 		\begin{split}
 			\varphi: \quad & \overline{\mathbb{B}}^{n+1}\longrightarrow \overline{\mathbb{R}}^{n+1}_+:=\{(y',y_{n+1})\in\mathbb{R}^{n+1}:y_{n+1}>0\}
 			\\& (x',x_{n+1}) \longmapsto \frac{2x'+(1-|x'|^2-x_{n+1}^2)E_{n+1}}{|x'|^2+(x_{n+1}-1)^2}:=(y',y_{n+1})=y,
 		\end{split}
 	\end{equation*}
 	where $ {x}=(x',x_{n+1})$ with $x'=(x_1,\ldots,x_n)\in\mathbb{R}^n$. Note that $\varphi$ maps $\mathbb{S}^n\setminus \{e\}$ to $\partial \mathbb{R}^{n+1}_+,$ and 
 	\begin{equation*}
 		\varphi^*(\delta_{\mathbb{R}^{n+1}_+})=\frac{4}{\left(|x'|^2+(x_{n+1}-1)^2\right)^2} \delta_{\mathbb{B}^{n+1}}:=e^{-2w}\delta_{\mathbb{B}^{n+1}},
 	\end{equation*}
 	which means that $\varphi$ is a conformal transformation from $\left(\overline{\mathbb{B}}^{n+1},\delta_{\mathbb{B}^{n+1}}\right)$ to $\left(\overline{\mathbb{R}}^{n+1}_+,\delta_{\mathbb{R}^{n+1}_+}\right)$. 
Since $\left(\overline{\mathbb{B}}^{n+1},\delta_{\mathbb{{B}}^{n+1}}\right)$ and $\left(\overline{\mathbb{R}}^{n+1}_+,(\varphi^{-1})^*(\delta_{\mathbb{R}^{n+1}_+})\right)$ are isometric, a proper embedded hypersurface $\Sigma_t= {x}(\overline{\mathbb{B}}^n,t)$ in $\left(\overline{\mathbb{B}}^{n+1},\delta_{\mathbb{{B}}^{n+1}}\right)$ can be identified as $\widetilde{\Sigma}_t:= {y}(\overline{\mathbb{B}}^n,t)$ in $\left(\overline{\mathbb{R}}^{n+1}_+,(\varphi^{-1})^*(\delta_{\mathbb{R}^{n+1}_+})\right)$, 
 	where $ {y}:=\varphi\circ {x}$.
	
	As shown in \cite{WX2},  $\varphi_*(X_e)=y$. Therefore, since $\<X_e, \nu\> >0$ on $\S_t$, we see
	$\widetilde{\Sigma}_t\subset \left(\overline{\mathbb{R}}^{n+1}_+,(\varphi^{-1})^*(\delta_{\mathbb{R}^{n+1}_+})\right)$  can be written  as a radial graph over $\overline{\mathbb{S}}^{n}_+$, that is, there exists some positive function $\rho(z',t)$ defined on  $\overline{\mathbb{S}}^{n}_+\times [0,T^*)$ such that 
 	\begin{eqnarray*}
 	y=\rho (z'(p,t),t)z'(p,t),
 	\end{eqnarray*}where $(p,t)\in \overline{\mathbb{B}}^n\times[0,T^*)$. Denote $z':=(\beta,\xi)\in [0,\frac{\pi}{2}]\times \mathbb{S}^{n-1}$ and $u:=\log \rho$, then the	flow  \eqref{flow with capillary} is equivalent to   the following scalar parabolic equation  on $\overline{\mathbb{S}}^n_+$, see \cite{WX2} for a detailed computation.
 	 \begin{eqnarray}\label{scalar heat eq}
\begin{cases}
	 \partial_t u =\frac{v}{\rho e^w}\tilde{f} &\quad \text{ in }\mathbb{S}^n_+\times [0,T^*),\\
	\nabla_{\p_\beta} u =\cos\theta v &\quad \text{ in } \p\mathbb{S}^n_+\times [0,T^*),\\
	u(\cdot,0) =u_0(\cdot)&\quad \text{ on }\mathbb{S}^n_+.
\end{cases}
 	 \end{eqnarray}
Here $v:=\sqrt{1+|\n u|^2}$, $u_0=\log \rho_0$,  $\rho_0$ is the corresponding quantity for $\S_0$ under the transformation $\varphi$, $\tilde{f}$ is the corresponding quantity of $f$ under the transformation $\varphi$, and $\p_\beta$ is the unit outward normal of $\p\mathbb{S}^n_+$ on $\overline{\mathbb{S}}^n_+$.
 	 
From the a priori estimates  in Corollaries \ref{convexity preserve}, \ref{star-shaped preserve} and \ref{curvature bound}, we see that $u$ is uniformly bounded in $C^2(\mathbb{S}^n_+\times [0,T^*))$ and the scalar equation in \eqref{scalar heat eq} is uniformly parabolic.
Note that $|\cos\theta|<1$, the boundary value condition in \eqref{scalar heat eq} satisfies	the uniformly oblique property,  from the standard parabolic theory  (see e.g. \cite[Theorem 6.1, Theorem 6.4 and Theorem 6.5]{Dong}, also  \cite[Theorem 5]{Ura}  and \cite[Theorem 14.23]{Lie}), we conclude the uniform $C^\infty$-estimates and the long-time existence of solution to \eqref{scalar heat eq}.
 \end{proof}

\begin{prop}\label{converge limit}
$x(\cdot, t)$ smoothly converges to a uniquely determined spherical cap around $a$ with $\theta$-capillary boundary, as $t\to\infty$.		
\end{prop}
\begin{proof}
Recall $ W_{1,\theta}(\widehat{\S_t})$ is non-decreasing. Precisely,
	\begin{eqnarray*}
	\p_t  W_{1,\theta}(\widehat{\S_t})=\frac{1}{n+1}\int_{\S_t} \left(\frac{\sigma_1\sigma_{n-1}}{n\sigma_n}
	-n\right)\langle x+\cos\theta  \nu,e\rangle \geq 0.
\end{eqnarray*} Since we have long time existence and  uniform $C^\infty$-estimates, we obtain 
\begin{eqnarray*}
\int_{\S_t} \left(\frac{\sigma_1\sigma_{n-1}}{n\sigma_n}
-n\right)\langle x+\cos\theta  \nu,e\rangle \to 0, \hbox{ as }t\to +\infty.
\end{eqnarray*}Hence from the equality characterization of the Newton-MacLaurin inequality $\frac{\sigma_1\sigma_{n-1}}{n\sigma_n}\geq n$, we see that any convergent subsequence must converge to a spherical cap. 
Next we show that any limit of a convergent subsequence  is uniquely determined, which implies the flow smoothly converges to a unique spherical cap. We shall use the argument in \cite{SWX}.

Note that we have proved that $x(\cdot,t)$ subconverges smoothly to a capillary boundary spherical cap $ C_{\theta,r_\infty}(e_\infty)$.  Since $ W_{n,\theta}$ is preserved along the flow \eqref{flow with capillary}, the radius $r_\infty$ is independent of the choice of the subsequence of $t$.
We next show in the following that $e_\infty=e$. Denote $r(\cdot,t)$ be the radius of the unique spherical cap  centered at $e$ with contact angle $\theta$ passing through the point $x(\cdot,t)$. 
Due to the
spherical barrier estimate, i.e. Proposition \ref{barrier est}, we know
\begin{eqnarray*}
	r_{\max}(t):=\max r(\cdot,t)=r(\xi_t,t),
\end{eqnarray*}  is non-increasing with respect to $t$, hence the limit $\lim\limits_{t\to +\infty} r_{\text{max}}(t)$ exists. Next we claim that
\begin{eqnarray}\label{max limit}
	\lim_{t\to +\infty} r_{\max}(t)=r_\infty.
\end{eqnarray}
We prove this claim by contradiction. Suppose \eqref{max limit} is not true, then there exists $\varepsilon>0$ such that
\begin{eqnarray}\label{lim of radius}
	r_{\max}(t)> r_\infty +{\varepsilon}, \hbox{ for } t \hbox{ large enough.}
\end{eqnarray} 

By definition, $r(\cdot,t)$ satisfies
\begin{eqnarray}\label{spherical cap}
	2\langle x,e\rangle \sqrt{r^2+2r\cos\theta+1}=|x|^2+2r\cos\theta+1.
\end{eqnarray}
By taking the time derivative for \eqref{spherical cap}, we get
\begin{eqnarray*}
&&\left(\frac{(r+\cos\theta)\langle x,e\rangle}{\sqrt{r^2+2r\cos\theta+1}}-\cos\theta\right) \p_t r
\\&=& \< \p_t x, x- \sqrt{r^2+2r\cos\theta+1}e\>= \< f\nu+T, x- \sqrt{r^2+2r\cos\theta+1}e\>.
\end{eqnarray*}
We evaluate at $(\xi_t,t)$. Since $\S_t$ is tangential to $C_{\theta, r}(e)$ at $x(\xi_t, t)$,
we have
$$\nu_{\S_t}(\xi_t, t)=\nu_{\p C_{\theta, r}(e)}(\xi_t, t)= \frac{1}{r}\left(x- \sqrt{r^2+2r\cos\theta+1}e\right).$$
Thus we deduce
\begin{eqnarray}
&&\left(\frac{(r_{\max}+\cos\theta)\langle x,e\rangle}{\sqrt{r_{\max}^2+2r_{\max}\cos\theta+1}}-\cos\theta\right)\frac{d}{dt}r_{\max}
= r\left(\frac{\langle x+\cos\theta  \nu, e\rangle }{F}-\langle X_e,\nu\rangle\right).\nonumber\\
\label{r-max}
\end{eqnarray}
We first claim that there exists $\delta_1>0$ such that \begin{eqnarray}\label{claim1}
\frac{(r_{\max}+\cos\theta)\langle x,e\rangle}{\sqrt{r_{\max}^2+2r_{\max}\cos\theta+1}}-\cos\theta\ge\delta_1.
\end{eqnarray}
This follows directly from \eqref{spherical cap}. In fact, from \eqref{spherical cap}, we see
\begin{eqnarray}\label{spherical cap1}
	\langle x,e\rangle^2(r^2+2r\cos\theta+1)=\frac{(|x|^2+1)^2}{4}+r\cos\theta(|x|^2+1)+r^2\cos^2\theta.
\end{eqnarray}
By using \eqref{spherical cap} again and \eqref{spherical cap1}, we have
\begin{eqnarray*}
&&(r+\cos\theta)\langle x,e\rangle-\cos\theta \sqrt{r^2+2r\cos\theta+1}
\\&=&(r+\cos\theta)\langle x,e\rangle-\cos\theta \frac{|x|^2+2r\cos\theta+1}{2\langle x,e\rangle}
\\&=&\frac{1}{r \langle x,e\rangle}\left[r(r+\cos\theta)\langle x,e\rangle^2-r^2\cos^2\theta-r\cos\theta \frac{|x|^2+1}{2}\right]
\\&=&\frac{1}{r \langle x,e\rangle}\left[\frac{(|x|^2+1)^2}{4}+r\cos\theta (|x|^2+1)-r\cos\theta \langle x,e\rangle^2-\langle x,e\rangle^2- r\cos\theta \frac{|x|^2+1}{2}\right]
\\&=&\frac{1}{r \langle x,e\rangle}\left[\frac{(|x-e||x+e|)^2}{4}+r\cos\theta\frac{|x-e|^2}{2}\right].
\end{eqnarray*}
This yields \eqref{claim1}.

Since the spherical caps $  C_{\theta,r_{\max} }(e)$ are the static solution to \eqref{flow with capillary} and $x(\cdot,t)$ is tangential to $C_{\theta,r_{\max} }(e)$ at $x(\xi_t,t)$, we see from \eqref{static-eq}
\begin{eqnarray}\label{xeq44}
	\frac{\langle x+\cos\theta\nu,e\rangle}{\langle X_e,\nu\rangle }\Big|_{x(\xi_t,t)}= \frac{\langle x+\cos\theta\nu,e\rangle}{\langle X_e,\nu\rangle }\Big|_{C_{\theta,r_{\max} }(e)}=\frac{1}{r_{\max}(t)}.
\end{eqnarray}
Since $x(\cdot,t)$ converges to $C_{\theta,r_\infty  }(e_\infty)$ and $r_\infty$ is uniquely determined, we have
\begin{eqnarray*}
	F=\frac{n\sigma_n}{\sigma_{n-1}}\to \frac{1}{r_\infty} \quad \text{uniformly},
\end{eqnarray*}as $t\to+\infty$. 
Thus there exists $T_0>0$ such that 
\begin{eqnarray*}
	\frac{1}{F} - {r_\infty} <\frac{\ep}{2},
\end{eqnarray*}
and hence 
\begin{eqnarray*}
	\frac{1}{F} - {r_{\max}(t)} <-\frac{\ep}{2},
\end{eqnarray*}for all $t>T_0$.
Taking into account of \eqref{xeq44}, we see
\begin{eqnarray*}
	\left(\frac{1}{F} - \frac{\langle x+\cos\theta\nu,e\rangle}{\langle X_e,\nu\rangle }\right)\Big|_{x(t,\xi_t)} <-\frac{\ep}{2},
\end{eqnarray*}for all $t>T_0$.
Finally, we conclude from \eqref{r-max} that there exists some $C>0$ such that $$\frac{d}{dt}r_{\max}\le -C\ep.$$
This is a contradiction  to the fact  that $\lim\limits_{t\to +\infty}\frac{d}{dt} r_{\max}=0$, hence the claim \eqref{max limit} is true. Similarly, we can obtain that \begin{eqnarray}\label{min limit}
	\lim_{t\to +\infty} r_{\min}(t)=r_\infty.
\end{eqnarray}Combining this with claim \eqref{max limit}, we know that $\lim\limits_{t\to\infty} r(t,\cdot)=r_\infty$, which is a constant.  This implies any limit of a convergent subsequence is the spherical cap arounnd $e$, the claimed uniqueness. We complete the proof of Proposition \ref{converge limit}.
\end{proof}
Combining Propositions \ref{global existence} and \ref{converge limit}, we get the assertion in Theorem \ref{flow-result}.

\subsection{Proof of Theorem \ref{thm 1}}\

With the preparations above, using the same approach as in \cite[Section 4]{SWX}, we can  prove the main result, i.e. Theorem \ref{thm 1}. 

\begin{proof}[\textbf{Proof of Theorem \ref{thm 1}}]

	Firstly, for given $\theta\in(0,\frac{\pi}{2})$, we claim that 
\begin{eqnarray*}
	{	f_{k}}(r):= W_{k}(\widehat{C_{ \theta,r}(e)}),
\end{eqnarray*}is strictly increasing with respect to $r>0$.
To prove this claim, given $r_0>0$,  consider the hypersurface flow $y:\overline{\BB}^n\times [0,+\infty)\to \overline{\BB}^{n+1}$ given by
\begin{eqnarray*}
	\begin{cases}
		\p_t y =X_e, \quad \text{ on }\overline{\BB}^n\times [0,+\infty) \\
		y(\cdot,0)=C_{\theta, r_0}(e)\quad \text{ on }\overline{\BB}^n.
	\end{cases}
\end{eqnarray*}
The flow hypersurfaces of this flow are $C_{\theta,r(t)}(e)$, where the radius $r(t)$ is some strictly increasing function. In fact, $r(t)$ satisfies  
\begin{eqnarray*}
r'(t)= r\sqrt{r^2+2r\cos\theta+1},
\end{eqnarray*}with $r(0)=r_0$. 
Hence from Theorem \ref{fvf}, we have
\begin{eqnarray*}
	\p_t  W_{k,\theta}(\widehat{ {C}_{\theta,r(t)}(e)})=\frac{k}{n+1}\binom{n}{k-1}^{-1}\int_{C_{\theta,r(t)}(e)} \sigma_k \langle X_e,\nu\rangle dA_t>0,
\end{eqnarray*}where the last inequality follows from the strictly convex of the spherical caps, hence it yields that $f_{k}(r)$ are strictly increasing functions, which follows the claim.

Assume that $\S$ is strictly convex. By evolving along the flow \eqref{flow with capillary} with initial strictly hypersurface $\S $, we denote  its convergence limit be the spherical cap given by  $C_{\theta,r_\infty}(e)$, then we have
\begin{eqnarray*}
	 W_{n,\theta}(\widehat{\S})&=	& W_{n,\theta}(\widehat{C_{\theta,r_\infty}(e)})=f_{n}(r_\infty) \\&=&f_{n}\circ f^{-1}_{k}\circ f_{k}(r_\infty)\geq f_{n}\circ f^{-1}_{k} (  W_{k,\theta}(\widehat{\S})),
\end{eqnarray*} 
moreover, equality holds iff $\S$ is a spherical cap. Also, 
\begin{eqnarray*}
	 W_{n+1,\theta}(\widehat{\S})&=& W_{n+1,\theta}(\widehat{C_{\theta,r_\infty}(e)}). \end{eqnarray*} 
Note that when $r\to 0$, $C_{\theta, r}(e)$ is more and more close to the cap $C^{\RR^{n+1}_+}_{\theta, r}$ of radius $r$ in the half space $\RR^{n+1}_+$ with $\theta$-capillary boundary. 
Then $$(n+1)\lim_{r\to 0}W_{n+1,\theta}(\widehat{C_{\theta,r}(e)})=\lim_{r\to 0}\int_{C^{\RR^{n+1}_+}_{\theta, r}} H_n= \frac{\omega_n}{2 }I_{\sin^2\theta}(\frac{n}{2},\frac{1}{2}).$$
The second equality follows from the simple fact that $$\left|C^{\RR^{n+1}_+}_{\theta, r}\right|=\left|C^{\RR^{n+1}_+}_{\theta, 1}\right|r^n= \frac{\omega_n}{2 }I_{\sin^2\theta}(\frac{n}{2},\frac{1}{2}) r^n.$$ 

When $\S$ is convex but not strictly convex, the inequality \eqref{af ineq} and the equality \eqref{gbc} follows by approximation. The equality characterization in \eqref{af ineq}  can be proved  similar to \cite{SWX} Section 4, by using an argument of \cite{GL09}. We omit the details here.
\end{proof}

In particular, we have following simplified expansion formulas for a convex hypersurface $\S$ with $\theta$-capillary boundary in $\bar{\BB}^{n+1}$.
 \begin{enumerate}

\item 	For $n=2$, 
	\begin{align}\label{topo invariant}
		3 W_{3,\theta}(\widehat{\S})= \int_\S H_2dA-\cos\theta|\widehat{\p\S}|+\sin\theta|\p\S|  = {2\pi} (1- \cos\theta ),
	\end{align}due to $
I_{\sin^2\theta}(1,\frac{1}{2}) =1-\cos\theta.$
\item 
For $n=3$, 
\begin{align*}
4W_{4,\theta}(\widehat{\S})
	&=  \int_\S H_3dA+ \sin^2\theta\int_{\p\S} H_1^{\mathbb{S}^3}+|\widehat{\p\S}|-\sin\theta \cos\theta |\p\S| 
\\&= {2\pi} (\theta-\sin\theta\cos\theta),
\end{align*} due to  $
I_{\sin^2\theta}(\frac{3}{2},\frac{1}{2})=\frac{2\theta-\sin(2\theta)}{\pi}.$

\item For $n=4$,  
\begin{eqnarray*}
	5 W_{5,\theta}(\widehat{\S})&=&  \int_\S H_4dA+\sin^3\theta \int_{\p\S} H_2^{\SS^4}-\frac{3}{2}\cos\theta\sin^2\theta \int_{\p\S} H_1^{\SS^4}\\&&-\cos\theta (1+\frac{1}{2}\sin^2\theta) |\widehat{\p\S}|+\sin\theta (1-\frac{\sin^2\theta}{3}) |\p\S| 
	 \\&=& \frac{2\pi^2}{3}(2-2\cos\theta-\cos\theta \sin^2\theta ).
\end{eqnarray*}
 \end{enumerate}
\begin{rem}
For $n=2$, we can use another viewpoint to show \eqref{topo invariant} directly. From Gauss-Bonnet formula, we see
	\begin{eqnarray*}
		2\pi &=&\int_\S KdA+\int_{\p\S}\kappa_g ds =\int_\S KdA+\int_{\p\S}(\cos\theta\widehat{\k}+\sin\theta) ds
		\\&=&\int_\S K dA+\sin\theta |\p\S|+\cos\theta \big(2\pi -|\widehat{\p\S}|\big),
	\end{eqnarray*}where $K=H_2$ is the Gauss curvature, $\k_g$ and $\widehat{\k}$ denote the geodesic curvature of $\p\S$ in $\S$ and $\widehat{\p\S} $ respectively. From $\mu=\cos\theta \overline{\nu}+\sin\theta \overline{N}$ implies $\k_g=\cos\theta\widehat{\k}+\sin\theta$.
	Hence it follows
	\begin{eqnarray*}
		&&3W_{3,\theta}(\widehat{\S})= \int_\S KdA-\cos\theta|\widehat{\p\S}|+\sin\theta|\p\S|  
		=2\pi(1- \cos\theta ). 
	\end{eqnarray*}
\end{rem}

\

\noindent\textbf{Acknowledgment:} LW is supported by project funded by China 
Postdoctoral Science Foundation (No. 2021M702143) and NSFC (Grant No. 12171260). CX is supported by NSFC (Grant No. 11871406). Parts of this work was done while LW was visiting the Tianyuan Mathematical Center in Southeast China and school of mathematical sciences at Xiamen University under the support of NSFC (Grant No. 12126102). He would like to express his deep gratitude to the center for its hospitality. We thank the anonymous referee for his/her careful reading and critical comments.

\

\end{document}